\documentclass[a4paper,fleqn,11pt]{cas-sc}

\usepackage[authoryear]{natbib}

\usepackage{setspace}
\usepackage{fontawesome5}
\usepackage{soul}

\usepackage{bm}
\usepackage[utf8]{inputenc}
\usepackage{amsmath}
\usepackage{amsfonts}
\usepackage{amssymb}
\usepackage{multicol}
\usepackage{multirow}
\usepackage{adjustbox,lipsum}
\usepackage{pgf,tikz}
\usepackage{graphicx}
\usepackage{epstopdf}
\usepackage{rotating}
\usepackage{booktabs}
\usepackage{longtable}

\usepackage{caption}
\usepackage{lscape, comment}
\usepackage{amsthm} % para teoremas
\usepackage[linesnumbered, lined, boxed, commentsnumbered, ruled, vlined]{algorithm2e}
\usepackage{tablefootnote}

\usepackage{subcaption}
\usepackage{indentfirst}
\usepackage[dvipsnames]{xcolor}
\usepackage{verbatim}
\usepackage{pdflscape}   % mejor que lscape: rota la página en el visor PDF
\usepackage{adjustbox}   % para escalar y centrar
\usepackage{float}       % para usar [H]
\usepackage{graphicx}    % necesario para \rotatebox

\usepackage[pagewise]{lineno}
\usepackage{caption} 

\DeclareCaptionFont{casstyle}{\sffamily\small}
\theoremstyle{plain}
\newtheorem{lemma}{Lemma}
\newtheorem{prop}{Proposition}
\newtheorem{coro}{Corollary}

\newtheorem{example}{Example}
\newtheorem*{example*}{Example}

\theoremstyle{remark}
\newtheorem{remark}{Remark}

\allowdisplaybreaks

\usepackage{enumitem}
\setlist[enumerate,1]{label=\arabic*.}
\setlist[enumerate,2]{label*=\arabic*.}
\setlist[enumerate,3]{label*=\arabic*.}

\date{\today}

\def\tsc#1{\csdef{#1}{\textsc{\lowercase{#1}}\xspace}}
\tsc{WGM}
\tsc{QE}
\tsc{EP}
\tsc{PMS}
\tsc{BEC}
\tsc{DE}
\usepackage{tikz}
\usepackage{pgfplots}
\pgfplotsset{compat=1.18}

\begin{document}
\let\WriteBookmarks\relax
\def\floatpagepagefraction{1}
\def\textpagefraction{.001}

\shorttitle{Mixture Linear Ordering Problem}

\shortauthors{Aledo, Domínguez, Jaime-Alcántara, and Landete}

\title[mode=title]{Uncovering latent consensus in heterogeneous populations: The Mixture Linear Ordering Problem}
%\title[mode=title]{Working with heterogeneous preference groups: The Mixture Linear Ordering Problem} 

\author[uclm]{Juan A. Aledo}[orcid=0000-0003-1786-8087]
\ead{JuanAngel.Aledo@uclm.es}

\author[um]{Concepci\'on Dom\'inguez}[orcid=0000-0002-9046-4997]
\ead{concepcion.dominguez@um.es}

\author[umh]{Juan~de Dios Jaime-Alc\'antara}[orcid=0009-0004-4841-1295]
\ead{jjaime@umh.es}
\cormark[1]

\author[umh]{Mercedes Landete}[orcid=0000-0002-5201-0476]
\ead{landete@umh.es}

\address[uclm]{Departamento de Matemáticas, Universidad de Castilla-La Mancha, Albacete 02071, Spain}
\address[um]{Departamento de Estadística e Investigación Operativa, Universidad de Murcia, Murcia 30100, Spain}
\address[umh]{Departamento de Estadística, Matemáticas e Informática, Instituto Centro de Investigación Operativa, Universidad Miguel Hernández de Elche, Alicante 03202, Spain}

\cortext[cor1]{Corresponding author. Email: jjaime@umh.es}

\begin{abstract}
The classical linear ordering problem seeks a single ranking representing a given preference matrix. While suitable for homogeneous populations, it fails when observed preferences arise from several latent groups with distinct ranking patterns. To address this limitation, we introduce an extension partitioning the population into latent groups, each characterized by its own linear order, relative size, and preference structure. The observed matrix is then explained as the aggregate outcome of these group-specific preferences. We develop mixed-integer programming formulations, including a compact reformulation yielding a geometric interpretation within the linear ordering polytope. Because exact solutions become computationally demanding for larger instances, we propose a multi-start alternating-direction matheuristic iteratively updating group rankings and weights. Computational experiments on synthetically generated instances, matching sizes typical in preference aggregation scenarios, demonstrate the effectiveness of the exact approach  in successfully recovering the underlying groups. Furthermore, the proposed heuristic delivers high-quality solutions in substantially shorter times, occasionally improving upon the exact method's best incumbent in difficult instances within the imposed time limit.
\end{abstract}

% \includegraphics{figs/grabs.pdf}

%\item Research highlights item 1
%\item Research highlights item 2
%\item Research highlights item 3

% Keywords

\maketitle

\section{Introduction}

The Linear Ordering Problem (LOP) is a classical combinatorial optimization problem, first introduced by \cite{chenery1958international}, that seeks to determine a single linear order maximizing consistency with a given matrix of pairwise preferences or dominance relations. In preference aggregation settings, this model provides a natural and effective way to summarize observed comparisons into one global ranking. From a computational perspective, \cite{garey1979computers} established that the LOP is an NP-hard problem. Despite this complexity, the LOP has had a profound impact on several fields. Notably, Wassily Leontief was awarded the Nobel Prize in Economics for developing the input-output model \citep{leontief1936quantitative}, an economic framework whose structural analysis relies heavily on the LOP through the triangulation of matrices. Since then, the problem has found numerous real-world applications, ranging from voting aggregation and sports rankings to scheduling with precedences. A comprehensive review of these applications, alongside the structure of the linear ordering polytope and the exact and heuristic methods developed to solve the problem, is provided by \cite{marti2022exact}.

However, the core feature of the LOP of aggregating all information into a single global ranking can also become a significant limitation. When the underlying population is heterogeneous (comprising distinct subgroups with internally coherent but differing preference patterns), forcing a single consensus order often yields a ranking that primarily reflects the dominant group while masking minority structures. To address related disparities, a growing body of literature has explored fairness in rankings \citep[see, e.g.,][for an overview]{pitoura2022fairness}. These approaches typically focus on the fairness of the items being ranked; for instance, by proposing frameworks that allocate exposure aligned with an item's merit \citep{singh2018fairness}, or by employing constrained optimization to prevent the under- or over-representation of specific groups of items in the final ranking \citep{celis2018ranking, chakraborty2022fair}. While these methods successfully mitigate visibility and representation biases among candidates, they do not account for the structural diversity of the population generating the preferences. In heterogeneous preference aggregation, the challenge is not merely to adjust item exposure within a single ranking, but rather to explicitly model and preserve the latent preference structures of minority voter groups, preventing their internal coherence from being diluted by a global consensus.

To this end, our approach departs from the traditional paradigm of a single consensus ranking by proposing a model that provides multiple rankings to capture the underlying heterogeneity of the population. In the literature, the problem of finding several representative rankings has been primarily addressed through clustering techniques, such as mixture models or partition methods specifically designed for ranking data \citep{murphy2003mixtures, lu2014effective}. While these methods aim to reflect different group preferences, a fundamental limitation is that they typically require access to the individual preferences of every voter to perform the clustering. In contrast, our approach tackles a more challenging and data-efficient scenario: we aim to extract these latent ranking structures and their relative sizes relying exclusively on the aggregated pairwise preference matrix, following the line of \cite{aledo2025consensus}. By operating with this reduced level of information, our method is able to discover and preserve distinct preference patterns without the need for the full disaggregated profile of the population. This perspective shifts the focus of fairness from the items being ranked to the voters themselves, ensuring that the internal coherence of different population groups is represented even when only collective dominance data is available.

In many real-world applications, individual voter profiles are unavailable or naturally non-existent, leaving the aggregated pairwise matrix as the sole source of information. This is the case of the LOP, as commented above, but also of other rank aggregation frameworks such as the Optimal Bucket Ordering Problem (OBOP), which was initially formalized by \cite{gionis2006algorithms} and has been further developed in recent research \citep{aledo2018approaching, aledo2025consensus}. Such reliance on aggregate data frequently arises due to strict privacy policies and data anonymization in web search and e-commerce logs (where only aggregate click-through comparisons are retained), in crowdsourcing platforms based on micro-tasks, or in sports tournaments where the basic unit of information is an aggregate matrix of head-to-head outcomes. Motivated by this challenge, we introduce a generalization of the classical LOP, which we term the \textit{Mixture Linear Ordering Problem} (MLOP). This model is designed to analyze any aggregated preference matrix to discover and extract potential latent heterogeneous profiles, without requiring prior knowledge of whether such subgroups actually exist in the population.

The richness of the aggregated pairwise matrix as a source of information has been explored from several angles in recent literature. On the one hand, \cite{aparicio2020linear} introduced the LOP of Sets, which leverages a ranking of individual elements to derive the closest consistent ranking of the groups of a given partition, measured in terms of the Kendall-$\tau$ distance. In the same spirit, \cite{Labbe2023cluster} proposed a framework to calculate an adequate order of items/sets when an order of sets/items is available. On the other hand, the Multiple LOP \citep{cameron2021linear,anderson2022fairness,benito2026multiple} addresses the fact that the aggregated matrix often admits a large number of equally optimal solutions with significantly different orderings, proposing models to characterize this solution set and inform the user about the representativeness of any single optimal ranking. Both lines of work underscore a common insight: a single ranking extracted from an aggregated matrix may fail to capture the full informational content it encodes. The present work pursues this idea further by asking whether the matrix itself can be decomposed into a mixture of latent preference structures, each associated with a coherent linear order and a relative population weight.

To formally address the MLOP, we first present a natural mixed-integer programming formulation. However, to enhance its practical resolution, we subsequently develop a more efficient compact reformulation. This alternative mathematical representation is particularly noteworthy as it opens the door to an elegant geometric study of the solution space. Specifically, it allows us to recast the MLOP purely in geometric terms: finding the optimal mixture profile is formally equivalent to minimizing the $L_1$ distance from a distance from a point representing the aggregated data to a specific subset of the linear ordering polytope. This perspective not only provides a novel and visually intuitive interpretation of the problem, but also bridges the gap between combinatorial rank aggregation and geometric optimization, shedding light on the underlying structure of the solutions. Regarding its complexity, the MLOP is inherently NP-hard. This is because it constitutes a generalization of the classical LOP, which is recovered as a special case when only a single group  is considered. Since any instance of the classical LOP can be trivially embedded into an MLOP instance, and the LOP is known to be NP-hard, the MLOP naturally inherits this computational intractability and is at least as difficult to solve. To address this challenge, we develop a scalable heuristic algorithm based on the alternating direction method described in \cite{cattaruzza2024exact}. This approach allows us to efficiently navigate the complex solution space and obtain high-quality approximations in reasonable computational times. Finally, we demonstrate the effectiveness of our exact and heuristic approaches through extensive computational experiments on custom-generated synthetic instances, alongside a practical case study using a real-world preference dataset.

The remainder of this paper is organized as follows. Section~\ref{sec:prob_def} formally defines our problem; it begins with a mathematical description of the classical LOP, discusses the motivation for seeking multiple preference groups, and formally introduces the MLOP. Section~\ref{sec:formulation} presents our mixed-integer programming models, introducing first an extended logical formulation followed by a more efficient compact reformulation. Additionally, it provides a geometric interpretation of the problem as minimizing the $L_1$ distance from a given point to a specific subset of the linear ordering polytope. Section~\ref{sec:heuristic} details the proposed heuristic algorithm based on the alternating direction method. Section~\ref{sec:exp_com} reports the computational experiments conducted on synthetic instances to evaluate the performance of both the exact and heuristic approaches. Section~\ref{sec:sushi} presents the practical application of our model using a sushi preference dataset. Finally, Section~\ref{sec:CFR} draws the main conclusions of this work and outlines directions for future research.

\section{Preliminaries and problem definition} \label{sec:prob_def}

The \emph{Linear Ordering Problem} (LOP) is a classical combinatorial optimisation problem whose goal is to determine a linear order of $n$ items that maximises the total weight of the pairwise comparisons consistent with that order (see \cite{marti2022exact}). Using the notation of \cite{ceberio2015linear}, the problem is defined on a weight matrix $C=(c_{rs})_{n\times n}$, where each coefficient $c_{rs}$ represents the preference of item $r$ over item $s$. In order to ensure consistency with the mixture framework introduced later, we consider normalised preference matrices satisfying $c_{rs}+c_{sr}=1$ for all distinct $r,s\in[[n]]$, where $[[n]]=\{1,2,\dots,n\}$. The elements of the main diagonal are left undefined (and represented by dashes in the examples), since the specific value assigned to them has no impact on the objective function.

Given a permutation $\sigma=(\sigma_1,\sigma_2,\dots,\sigma_n)$ of the $n$ items, where $\sigma_k$ denotes the item placed in position $k$, we denote its LOP value by
\[
    f(\sigma, C) = \sum_{r=1}^{n-1}\sum_{s=r+1}^n c_{\sigma_r\sigma_s},
\]
which corresponds to the total weight of all pairwise relations consistent with the order induced by $\sigma$. For convenience, we use the notation $r \succ s$ to denote that item $r$ precedes item $s$ in the ranking. The LOP seeks to find a permutation that maximises this quantity.

To formulate the problem as an integer program, we adopt the classical binary precedence variables introduced in \cite{marti2022exact}. We use the following decision variables:
\[
x_{rs} =
\begin{cases}
1, & \text{if item } r \text{ precedes item } s,\\[1mm]
0, & \text{otherwise},
\end{cases}
\qquad \forall\, r,s \in [[n]] : r \neq s.
\]
Using these variables, a standard integer programming formulation of the LOP is
\begin{subequations}\label{model:LOP}
    \begin{align}
    \max \quad & \sum_{r \, = \,1}^n \sum_{\substack{s \, = \, 1 \\ r \, \neq \, s}}^n c_{rs}\, x_{rs}
        \label{objLOP} \\[2mm]
    \text{s.t.}\quad
        & x_{rs} + x_{sr} = 1
        && \forall\, r,s \in [[n]] : r < s,
        \label{cons:pairsLOP} \\
        & x_{rs} + x_{st} + x_{tr} \le 2
        && \forall\, r,s,t \in [[n]] : r < s,\; r < t,\; s \neq t,
        \label{cons:transitivityLOP} \\
        & x_{rs} \in \{0,1\}
        && \forall\, r,s \in [[n]] : r \neq s.
        \label{cons:binaryLOP}
    \end{align}
\end{subequations}
The objective function \eqref{objLOP} is the total weight of all precedence relations consistent with the final ranking. Constraints \eqref{cons:pairsLOP} impose strict comparability between every pair of items, ensuring that for each pair exactly one precedence direction is selected. Constraints \eqref{cons:transitivityLOP} forbid directed 3-cycles in the precedence relation. Together with \eqref{cons:pairsLOP}, this ensures that the induced relation is transitive and therefore defines a linear order of the items. Finally, the binary restrictions \eqref{cons:binaryLOP} guarantee that the feasible region corresponds exactly to the set of all linear orders of the $n$ items.

While formulation \eqref{model:LOP} is standard and intuitive, it contains redundant information. Exploiting the comparability equations \eqref{cons:pairsLOP}, one can eliminate half of the precedence variables by substituting $x_{sr}=1-x_{rs}$ for all $r>s$, thereby obtaining an equivalent compact formulation involving only variables $x_{rs}$ for $r<s$. Since this reformulation is classical, we omit its explicit statement and retain it only as a conceptual reference for the compact model introduced later.

We consider the following toy instance as a running example to illustrate the problem throughout the paper.
\begin{example} \label{ex:LOP}
Consider the instance of the {\rm LOP} with $n = 4$ items and the following weight matrix:
\[
\begin{pmatrix}
- & 0.9 & 0.9 & 0.9 \\[0.5em]
0.1 & - & 0.5 & 0.9 \\[0.5em]
0.1 & 0.5 & - & 0.9 \\[0.5em]
0.1 & 0.1 & 0.1 & -
\end{pmatrix}
\]
Solving the classical {\rm LOP} for this instance yields the optimal ranking
\[
1 \succ 2 \succ 3 \succ 4.
\]
The objective value of this ranking is $5$.
\end{example}

\subsection{Motivation for heterogeneous groups}

The classical LOP produces a single ranking that aggregates all preference information into a unique linear order. While this approach is well suited to settings where the underlying population can be regarded as homogeneous, it may lead to unfair or unrepresentative outcomes when preferences differ across subgroups. In such situations, the observed preference matrix may reflect the coexistence of several distinct preference patterns, even though only an aggregated matrix is available. A single global solution may therefore misrepresent legitimate alternative structures simply because they correspond to a smaller portion of the population.

\begin{example}
Consider again the instance introduced in Example~\ref{ex:LOP}, where solving the classical {\rm LOP} yields the ranking $1 \succ 2 \succ 3 \succ 4$ with objective value $5$. To illustrate how heterogeneous preference structures can be hidden in an aggregated matrix, consider the following latent decomposition of the original matrix into two components weighted at $90\%$ and $10\%$:
\[
\begin{pmatrix}
- & 0.9 & 0.9 & 0.9 \\[0.5em]
0.1 & - & 0.5 & 0.9 \\[0.5em]
0.1 & 0.5 & - & 0.9 \\[0.5em]
0.1 & 0.1 & 0.1 & -
\end{pmatrix}
=
0.9\cdot
\begin{pmatrix}
- & 1 & 1 & 1 \\[0.5em]
0 & - & \dfrac{5}{9} & 1 \\[0.5em]
0 & \dfrac{4}{9} & - & 1 \\[0.5em]
0 & 0 & 0 & -
\end{pmatrix}
+
0.1\cdot
\begin{pmatrix}
- & 0 & 0 & 0 \\[0.5em]
1 & - & 0 & 0 \\[0.5em]
1 & 1 & - & 0 \\[0.5em]
1 & 1 & 1 & -
\end{pmatrix}.
\]

We now solve the {\rm LOP} separately on each component matrix. 
For the first component, which carries $90\%$ of the weight, the optimal ranking is
\[
1 \succ 2 \succ 3 \succ 4,
\]
with objective value $\tfrac{50}{9}$. For the second component, weighted at $10\%$, the optimal ranking is
\[
4 \succ 3 \succ 2 \succ 1,
\]
with objective value $6$. The aggregated performance obtained by combining the two latent rankings according to their weights is
\[
0.9\cdot\tfrac{50}{9} + 0.1\cdot 6 = 5.4,
\]
which exceeds the classical {\rm LOP} value of $5$. This illustrates that an aggregated preference matrix may hide heterogeneous underlying structures that cannot be captured by a single linear order.
\end{example}

These observations motivate the introduction of a generalised framework that seeks to identify multiple latent preference structures whose combination explains the observed matrix. The aim is to preserve heterogeneous preference patterns rather than collapsing them into a single global ranking, thereby providing a richer and more representative description than the classical LOP.

\subsection{The Mixture Linear Ordering Problem}

The Mixture Linear Ordering Problem with $g$ groups, denoted by $\mathrm{MLOP}_g$, generalises the classical Linear Ordering Problem by modelling the observed preference matrix as the result of several latent preference structures. Rather than assuming that all individuals share a single homogeneous preference pattern, $\mathrm{MLOP}_g$ allows the observed pairwise comparison matrix to arise from up to $g$ underlying preference groups, each endowed with its own linear order and representing a certain proportion of the population. The aim is to identify a collection of latent group-specific rankings together with their associated weights so that the aggregated preference structure is explained in a coherent and representative manner.

Formally, let $n$ be the number of items and let $C = (c_{rs})_{n \times n}$ denote the observed (normalised) preference matrix, where $c_{rs}$ represents the proportion of individuals who prefer item $r$ to item $s$. We fix a maximum number of groups $g \ge 1$. A feasible solution to $\mathrm{MLOP}_g$ consists of the following components:

\begin{itemize}
    \item For each group $i \in [[g]]$, a linear order $\sigma^{i}$ of the $n$ items.

    \item For each group $i \in [[g]]$, a weight $\omega^i \ge 0$ representing the relative size of the group, satisfying
    \[
        \sum_{i=1}^g \omega^i = 1.
    \]

    \item For each group $i \in [[g]]$, a latent normalised preference matrix $C^{i} = (c_{rs}^{i})_{n \times n}$.
\end{itemize}

These latent matrices are not observed but are required to reproduce the observed matrix through the mixture constraint
\[
    C = \sum_{i=1}^{g} \omega^{i} \, C^{i}.
\]
In this way, the observed pairwise comparisons are interpreted as a convex combination of several underlying group-specific preference structures.

For each group $i$, the contribution to the objective function is the classical LOP score
\[
    f\bigl(\sigma^{i}, C^{i} \bigr)
    = \sum_{r=1}^{n-1} \sum_{s=r+1}^n c_{\sigma^{i}_r \sigma^{i}_s}^{i},
\]
which measures how well the latent matrix $C^i$ aligns with the order $\sigma^{i}$. 
The objective of $\mathrm{MLOP}_g$ is to maximise the weighted sum
\[
    F(\sigma^1,\dots,\sigma^g,C^1,\dots,C^g,\omega^1,\dots,\omega^g)
    = \sum_{i=1}^g \omega^i\, f\bigl(\sigma^{i},C^{i}\bigr),
\]
subject to the mixture constraint above.

Thus, $\mathrm{MLOP}_g$ seeks a collection of latent linear orders and weights whose convex combination reproduces the observed preference matrix $C$ while maximising the total LOP value induced by these latent structures.

A natural structural property of the model is that its optimal value is non-decreasing with respect to the number of groups allowed. Indeed:

\begin{lemma} \label{lem:creciente}
For every integer $g \ge 1$, the optimal value of $\mathrm{MLOP}_g$ is less than or equal to the optimal value of $\mathrm{MLOP}_{g+1}$.
\end{lemma}
\begin{proof}
Any feasible solution of $\mathrm{MLOP}_g$ can be extended to a feasible solution of $\mathrm{MLOP}_{g+1}$ with the same objective value by adding an additional group with weight $\omega^{g+1} = 0$, and assigning it an arbitrary valid preference matrix and an arbitrary linear order. 
\end{proof}

This monotonicity implies that an optimal solution of $\mathrm{MLOP}_g$ does not necessarily use all $g$ groups. If fewer than $g$ groups are sufficient to reproduce the observed matrix $C$, the remaining weights may simply be set to zero. Moreover, as will be shown later, the number of active groups in an optimal solution admits a finite upper bound, so increasing $g$ enlarges the feasible set without necessarily increasing the effective number of groups used at optimality.

\medskip

We illustrate how the problem operates in practice through our running example.

\begin{example}
Continuing from Example~\ref{ex:LOP}, for $g \ge 3$ groups, the following decomposition of the observed preference matrix is valid:
\[
\begin{pmatrix}
    - & 0.9 & 0.9 & 0.9 \\[0.5em]
    0.1 & - & 0.5 & 0.9 \\[0.5em]
    0.1 & 0.5 & - & 0.9 \\[0.5em]
    0.1 & 0.1 & 0.1 & - 
\end{pmatrix}
=
0.5 \,
\begin{pmatrix}
    - & 1 & 1 & 1 \\[0.5em]
    0 & - & 1 & 1 \\[0.5em]
    0 & 0 & - & 1 \\[0.5em]
    0 & 0 & 0 & -
\end{pmatrix}
+
0.4 \,
\begin{pmatrix}
    - & 1 & 1 & 1 \\[0.5em]
    0 & - & 0 & 1 \\[0.5em]
    0 & 1 & - & 1 \\[0.5em]
    0 & 0 & 0 & -
\end{pmatrix}
+
0.1 \,
\begin{pmatrix}
    - & 0 & 0 & 0 \\[0.5em]
    1 & - & 0 & 0 \\[0.5em]
    1 & 1 & - & 0 \\[0.5em]
    1 & 1 & 1 & -
\end{pmatrix}.
\]

A possible choice of rankings for the three groups is:
\[
1 \succ 2 \succ 3 \succ 4, 
\qquad
1 \succ 3 \succ 2 \succ 4,
\qquad
4 \succ 3 \succ 2 \succ 1.
\]

Each group-specific ranking attains an {\rm LOP} value of $6$. Since $\binom{4}{2} = 6$ is the maximum possible score in the classical {\rm LOP} with four items, their weighted combination necessarily achieves the global optimum. Therefore, using the permutations induced by the three chosen rankings, together with the matrices in the decomposition and their associated weights, provides an optimal solution to $\mathrm{MLOP}_g$ for any $g \ge 3$.
\end{example}

\section{Mixed-integer linear programming formulations} \label{sec:formulation}

In this section we develop mixed-integer programming formulations for the Mixture Linear Ordering Problem. We first introduce an extended formulation that explicitly models the latent preference structure associated with each group, providing a direct mathematical representation of the problem definition. We then derive a compact reformulation that eliminates the latent matrices and leads to a linearised model suitable for computation. Finally, exploiting the resulting absolute-value interpretation, we establish structural properties of the problem and derive bounds on the number of groups that may be active at optimality.

\subsection{Extended formulation with latent preference matrices}

We now introduce an extended formulation for the Mixture Linear Ordering Problem that closely follows the classical formulation of the Linear Ordering Problem presented in the previous section. The key idea is to associate with each group its own ranking variables together with a latent preference matrix whose weighted combination reproduces the observed preference matrix.

As established earlier for the normalized preference matrix, we assume $c_{rs} + c_{sr} = 1$ for all distinct $r,s \in [[n]]$, leaving the elements of the main diagonal undefined. To formulate this extended model, we introduce the following decision variables:
\[
x_{rs}^i =
\begin{cases}
1, & \text{if item } r \text{ precedes item } s \text{ in the ranking of group } i,\\
0, & \text{otherwise},
\end{cases}
\qquad
\forall\, r,s\in[[n]],\ i\in[[g]] : r\neq s,
\]
\[
y_{rs}^i = \text{portion of } c_{rs} \text{ explained by group } i,
\qquad
\forall\, r,s\in[[n]],\ i\in[[g]] : r\neq s,
\]
\[
\omega^i = \text{weight associated with group } i,
\qquad
\forall i\in[[g]].
\]

The resulting formulation is
\begin{subequations}\label{model:MLOP}
\begin{align}
\max \quad 
& \sum_{r=1}^n \sum_{\substack{s=1\\ r\neq s}}^n \sum_{i=1}^g
y_{rs}^i\,x_{rs}^i
\label{objMLOP}
\\
\text{s.t.}\quad
& x_{rs}^i+x_{sr}^i=1
&& \forall r,s\in[[n]],\ i\in[[g]] : r<s,
\label{cons:pairsMLOP}
\\
& x_{rs}^i+x_{st}^i+x_{tr}^i \le 2
&& \forall r,s,t\in[[n]],\ i\in[[g]] : r<s,\ r<t,\ s\neq t,
\label{cons:transMLOP}
\\
& \sum_{i=1}^g y_{rs}^i = c_{rs}
&& \forall r,s\in[[n]] : r\neq s,
\label{cons:sumcMLOP}
\\
& y_{rs}^i + y_{sr}^i = \omega^i
&& \forall r,s\in[[n]],\ i\in[[g]] : r<s,
\label{cons:pairweightMLOP}
\\
& x_{rs}^i \in \{0,1\}
&& \forall r,s\in[[n]],\ i\in[[g]] : r\neq s,
\label{cons:binaryMLOP}
\\
& y_{rs}^i \ge 0
&& \forall r,s\in[[n]],\ i\in[[g]] : r\neq s,
\label{cons:yposMLOP}
\\
& \omega^i \ge 0
&& \forall i\in[[g]].
\label{cons:wposMLOP}
\end{align}
\end{subequations}

Constraints \eqref{cons:pairsMLOP}--\eqref{cons:transMLOP} ensure that the variables $\bm x^i$ define a linear order for each group. Constraint \eqref{cons:sumcMLOP} guarantees that the sum of the group-specific preference matrices reproduces the observed preference matrix $C$. Finally, constraints \eqref{cons:pairweightMLOP} ensure that each group corresponds to a fraction $\omega^i$ of the population.

Although the weights $\omega^i$ do not appear explicitly in the objective function \eqref{objMLOP}, they are implicitly encoded in the variables $y_{rs}^i$. Furthermore, while the normalisation condition $\sum_{i=1}^g \omega^i = 1$ is not explicitly formulated, we will show in Proposition~\ref{prop:MLOP-feas} that it is naturally enforced by the model constraints. The following proposition shows that formulation \eqref{model:MLOP} correctly models $\mathrm{MLOP}_g$.

\begin{prop}\label{prop:MLOP-feas}
Let $(\bm x,\bm y,\bm\omega)$ be a feasible solution of \eqref{model:MLOP}. Then it induces a feasible solution $(\sigma^i,C^i,\omega^i)_{i\in[[g]]}$ of $\mathrm{MLOP}_g$.
\end{prop}

\begin{proof}
We first show how to extract the group rankings from $\bm x$. Fix $i\in[[g]]$ and define a binary relation $\prec_i$ on $[[n]]$ by declaring that $r\prec_i s$ if and only if $x_{rs}^i=1$. Constraint \eqref{cons:pairsMLOP} guarantees completeness and antisymmetry: for every pair $r\neq s$, exactly one of $x_{rs}^i$ and $x_{sr}^i$ equals $1$, so either $r\prec_i s$ or $s\prec_i r$, but not both. Constraint \eqref{cons:transMLOP} ensures transitivity of $\prec_i$. Hence $\prec_i$ is a strict total order on $[[n]]$, and therefore it is represented by a unique permutation $\sigma^i$ such that $x_{\sigma^i_r\,\sigma^i_s}^i=1$ for all $1\le r<s\le n$.

We next construct the latent preference matrices. Consider $i\in[[g]]$ with $\omega^i>0$ and define $c_{rs}^i:=y_{rs}^i/\omega^i$ for all $r\neq s$. Since $y_{rs}^i\ge 0$, we have $c_{rs}^i\ge 0$. Moreover, by \eqref{cons:pairweightMLOP} we have $y_{rs}^i+y_{sr}^i=\omega^i$ for all $r<s$, and dividing by $\omega^i$ yields $c_{rs}^i+c_{sr}^i=1$, which in particular implies $c_{rs}^i\le 1$ and $c_{sr}^i\le 1$. Thus $C^i=(c_{rs}^i)_{n\times n}$ is a valid normalised preference matrix (leaving the elements of the main diagonal undefined). If $\omega^i=0$, then \eqref{cons:pairweightMLOP} together with the non-negativity constraints \eqref{cons:yposMLOP} forces $y_{rs}^i=0$ for all $r\neq s$, so group $i$ plays no role and we may choose an arbitrary valid preference matrix $C^i$ and linear order $\sigma^i$.

It remains to verify the mixture constraint and the normalisation of the weights. Using \eqref{cons:sumcMLOP}, for each $r\neq s$ we have $c_{rs}=\sum_{i=1}^g y_{rs}^i$. Notice that the relation $y_{rs}^i = \omega^i c_{rs}^i$ holds for all $i \in [[g]]$: if $\omega^i > 0$ it follows from our definition of $c_{rs}^i$, and if $\omega^i = 0$ it holds trivially since $y_{rs}^i = 0$. Substituting this directly into the sum gives $c_{rs}=\sum_{i=1}^{g} \omega^i\, c_{rs}^i$,
which is precisely the mixture constraint $C=\sum_{i=1}^g \omega^i C^i$ entrywise. Finally, for any fixed pair $r\neq s$, summing \eqref{cons:pairweightMLOP} over $i$ and using \eqref{cons:sumcMLOP} yields
\[
\sum_{i=1}^g \omega^i=\sum_{i=1}^g (y_{rs}^i+y_{sr}^i)=\sum_{i=1}^g y_{rs}^i+\sum_{i=1}^g y_{sr}^i=c_{rs}+c_{sr}=1,
\]
so the weights are properly normalised.

To conclude, for any group $i \in [[g]]$,
\[
\omega^i f(\sigma^i,C^i)=\omega^i \, \sum_{r=1}^n \sum_{\substack{s=1\\ r\neq s}}^n c_{rs}^i \,  x_{rs}^i=\sum_{r=1}^n \sum_{\substack{s=1\\ r\neq s}}^n \omega^i \, c_{rs}^i \, x_{rs}^i=\sum_{r=1}^n \sum_{\substack{s=1\\ r\neq s}}^n y_{rs}^i \, x_{rs}^i,
\]
where we use the identity $y_{rs}^i=\omega^i c_{rs}^i$. Summing over $i=1,\dots,g$ yields exactly \eqref{objMLOP}, and the proof is complete.
\end{proof}

\begin{remark}
    The formulation \eqref{model:MLOP} presents a large number of symmetric solutions, since the groups are unlabeled and can be permuted without affecting feasibility or objective value. To reduce this symmetry, we can impose an ordering on the group weights,
\begin{align} \label{eq:sim}
    & \omega^i \ge \omega^{i+1}
    \qquad
    && \forall i\in[[g-1]],
\end{align}
which selects a canonical representative among equivalent solutions. 
\end{remark}

The bilinear terms $x_{rs}^i y_{rs}^i$ in \eqref{objMLOP} could be linearised directly. Furthermore, one could eliminate the redundancy of the model by exploiting the relations $y_{sr}^i = \omega^i - y_{rs}^i$ and $x_{sr}^i = 1 - x_{rs}^i$ to work only with pairs $r < s$. However, we do not pursue these straightforward modifications here. Instead, in the following section, we introduce a more elegant approach to consolidate the model into a truly compact formulation, which will be the one used for our computational study.

\subsection{Compact reformulation} \label{subsec:compact_ref}

The extended formulation \eqref{model:MLOP} introduces the variables $y_{rs}^i$
to describe how the observed matrix $C$ is decomposed into group-specific
components. However, these variables are not needed explicitly. The next lemma
shows that their contribution to the objective can be written directly in terms
of the weights $\omega^i$ and the precedence variables $x_{rs}^i$.

\begin{lemma}\label{lem:minstructure}
Let $(\hat{\bm{x}}, \hat{\bm{y}}, \hat{\bm{\omega}})$ be an optimal solution of
\eqref{model:MLOP}. Then, for any $r,s \in [[n]]$ with $r \neq s$,
\[
\sum_{i=1}^g \hat{x}_{rs}^i \hat{y}_{rs}^i
=
\min\Bigl\{c_{rs},\ \sum_{i=1}^g \hat{\omega}^i \hat{x}_{rs}^i\Bigr\}.
\]
\end{lemma}

\begin{proof}
Fix $(\hat{\bm{x}}, \hat{\bm{\omega}})$. Observe that for any unordered pair $\{r,s\} \subseteq [[n]]$ with $r \neq s$, the corresponding block of $2g$ variables $\{y_{rs}^i, y_{sr}^i\}, {i \in [[g]]}, $ only appears in constraints \eqref{cons:sumcMLOP}, \eqref{cons:pairweightMLOP}, and \eqref{cons:yposMLOP}. Crucially, none of these constraints link $y$-variables belonging to different unordered pairs. Therefore, once $\hat{\bm{\omega}}$ is fixed, we can reassign the values within each block $\{y_{rs}^i, y_{sr}^i\}_{i \in [[g]]}$ independently of the rest of the variables. As long as this local reassignment satisfies \eqref{cons:sumcMLOP}-\eqref{cons:yposMLOP} for the specific pair $\{r,s\}$, the overall feasibility of the solution is preserved.

Now, fix an unordered pair $\{r,s\}$ with $r \neq s$. Since $\sum_{i=1}^g y_{rs}^i = c_{rs}$ and
$0 \le y_{rs}^i \le y_{rs}^i + y_{sr}^i = \hat{\omega}^i$ for all $i$, it follows that any feasible assignment for $\bm{y}$
satisfies
\[
\sum_{i=1}^g \hat{x}_{rs}^i y_{rs}^i
= \sum_{i:\hat{x}_{rs}^i=1} y_{rs}^i
\le \min\Bigl\{\sum_{i=1}^g y_{rs}^i,\ \sum_{i:\hat{x}_{rs}^i=1} \hat{\omega}^i\Bigr\}
= \min\Bigl\{c_{rs},\ \sum_{i=1}^g \hat{\omega}^i \hat{x}_{rs}^i\Bigr\}.
\]
Consequently, this establishes the right-hand side as an upper bound. Next, we explicitly construct a feasible assignment for $\{y_{rs}^i, y_{sr}^i\}_{i \in [[g]]}$ to show that this bound is always attained.

If $c_{rs}=0$, then $\sum_{i=1}^g \hat{y}_{rs}^i = c_{rs} = 0$ and $\hat{y}_{rs}^i \ge 0$ imply
$\hat{y}_{rs}^i = 0$ for all $i$, and thus
\[
\sum_{i=1}^g \hat{x}_{rs}^i \hat{y}_{rs}^i = 0
= \min\Bigl\{c_{rs},\ \sum_{i=1}^g \hat{\omega}^i \hat{x}_{rs}^i\Bigr\}.
\]
So assume $c_{rs} > 0$ in what follows and define
\[
I_0(r,s) := \{i \in [[g]] : \hat{x}_{rs}^i = 0\}, \qquad I_1(r,s) := [[g]] \setminus I_0(r,s) = \{i \in [[g]] : \hat{x}_{rs}^i = 1\}.
\]

\smallskip
\noindent\textbf{Case 1:} $c_{rs} \le \sum_{i \in I_1(r,s)} \hat{\omega}^i$.
Set $\alpha := c_{rs} / \sum_{i \in I_1(r,s)} \hat{\omega}^i$ and define
\[
\tilde{y}_{rs}^i :=
\begin{cases}
0 & \text{if } i \in I_0(r,s), \\
\alpha\,\hat{\omega}^i & \text{if } i \in I_1(r,s),
\end{cases}
\qquad
\tilde{y}_{rs}^i := \hat{\omega}^i - \tilde{y}_{rs}^i \quad \forall i \in [[g]].
\]
Then, $0 \le \tilde{y}_{rs}^i \le \hat{\omega}^i$ and $\tilde{y}_{rs}^i + \tilde{y}_{sr}^i = \hat{\omega}^i$ for all $i$.
Moreover,
\[
\sum_{i=1}^g \tilde{y}_{rs}^i = \sum_{i \in I_1(r,s)} \tilde{y}_{rs}^i = \sum_{i \in I_1(r,s)} \alpha \, \hat{\omega}^i
= \alpha \sum_{i \in I_1(r,s)} \hat{\omega}^i = c_{rs},
\]
and therefore $\sum_{i=1}^g \tilde{y}_{rs}^i = \sum_{i=1}^g \hat{\omega}^i - \sum_{i=1}^g \tilde{y}_{rs}^i = 1 - c_{rs} = c_{sr}$.
Finally,
\[
\sum_{i=1}^g \hat{x}_{rs}^i \tilde{y}_{rs}^i
= \sum_{i \in I_1(r,s)} \tilde{y}_{rs}^i
= \sum_{i=1}^g \tilde{y}_{rs}^i
= c_{rs}
= \min\Bigl\{c_{rs}, \sum_{i=1}^g \hat{\omega}^i \hat{x}_{rs}^i\Bigr\}.
\]

\smallskip
\noindent\textbf{Case 2:} $c_{rs} > \sum_{i \in I_1(r,s)} \hat{\omega}^i$.
Since the preference matrix is normalised and the weights satisfy $\sum_{i=1}^g \hat{\omega}^i = 1$, we have $c_{sr} = 1 - c_{rs} < 1 - \sum_{i \in I_1(r,s)} \hat{\omega}^i = \sum_{i \in I_0(r,s)} \hat{\omega}^i$. Therefore the condition of Case~1 holds for the ordered pair $(s,r)$. Applying the construction of Case~1 to $(s,r)$, we obtain a feasible assignment for the block $\{\tilde{y}_{rs}^i, \tilde{y}_{rs}^i\}_{i \in [[g]]}$
such that $\tilde{y}_{rs}^i = 0$ for all $i \in I_0(s,r) = I_1(r,s)$. Hence, using $\tilde{y}_{rs}^i = \hat{\omega}^i - \tilde{y}_{sr}^i$, we obtain
\[
\sum_{i=1}^g \hat{x}_{rs}^i \, \tilde{y}_{rs}^i
= \sum_{i \in I_1(r,s)} (\hat{\omega}^i - \tilde{y}_{sr}^i)
= \sum_{i \in I_1(r,s)} \hat{\omega}^i
= \sum_{i=1}^g \hat{\omega}^i \, \hat{x}_{rs}^i
= \min\Bigl\{c_{rs}, \sum_{i=1}^g \hat{\omega}^i \, \hat{x}_{rs}^i\Bigr\}.
\]

Applying the above construction independently to all unordered pairs $\{r,s\}$ such that $r < s$ yields a feasible
$\tilde{\bm{y}}$ attaining the upper bound for every $r \neq s$. Since this bound holds for any feasible solution, and the objective function of \eqref{model:MLOP} maximizes the sum $\sum_{r \neq s} \sum_{i=1}^g \hat{x}_{rs}^i \, y_{rs}^i$, optimality implies that $\hat{\bm{y}}$ must also attain this maximum possible value. The claim follows.
\end{proof}

Applying Lemma~\ref{lem:minstructure} entrywise to model~\eqref{model:MLOP} allows the elimination of the variables $y_{rs}^i$, leading to the equivalent formulation:
\begin{subequations}\label{model:MLOP2}
\begin{align}
\max \quad
& \sum_{r=1}^n \sum_{\substack{s=1 \\ r\neq s}}^n
\min\Bigl\{c_{rs},\ \sum_{i=1}^g \omega^i\,x_{rs}^i\Bigr\}
\label{obj2}\\
\text{s.t.}\quad
& \eqref{cons:pairsMLOP}, \eqref{cons:transMLOP} \\
& \sum_{i=1}^g \omega^i = 1,
\label{cons:weightsum2}\\
& x_{rs}^i \in \{0,1\}
&& \forall r,s\in[[n]],\ \forall i\in[[g]]: r\neq s,
\label{cons:binary2}\\
& \omega^i \ge 0
&& \forall i\in[[g]].
\label{cons:wpos2}
\end{align}
\end{subequations}
Note that the normalization of weights \eqref{cons:weightsum2}, which was implicitly enforced by the constraints on $y_{rs}^i$ in the original model \eqref{model:MLOP}, must now be explicitly included as those variables have been removed. 

For distinct \(r,s \in [[n]]\), using \eqref{cons:pairsMLOP} and \eqref{cons:weightsum2}, we have \(\sum_{i=1}^g \omega^i x_{sr}^i = \sum_{i=1}^g \omega^i(1 - x_{rs}^i) = 1-\sum_{i=1}^g \omega^i x_{rs}^i\). Furthermore, since \(C\) is normalised, \(c_{sr}=1-c_{rs}\). These relations allow us to simplify the objective function by grouping pairs of items:
\begin{align*}
\min\Bigl\{c_{rs},\sum_{i=1}^g \omega^i x_{rs}^i\Bigr\} + \min\Bigl\{c_{sr},\sum_{i=1}^g \omega^i x_{sr}^i\Bigr\}
&= \min\Bigl\{c_{rs},\sum_{i=1}^g \omega^i x_{rs}^i\Bigr\} + \min\Bigl\{1-c_{rs}, 1-\sum_{i=1}^g \omega^i x_{rs}^i\Bigr\} \\
&= \min\Bigl\{c_{rs},\sum_{i=1}^g \omega^i x_{rs}^i\Bigr\} + 1 - \max\Bigl\{c_{rs},\sum_{i=1}^g \omega^i x_{rs}^i\Bigr\} \\
&= 1 - \Bigl|c_{rs}-\sum_{i=1}^g \omega^i x_{rs}^i\Bigr|.
\end{align*}

As mentioned in Section~2, we can eliminate the redundancy of the pairwise comparisons by focusing exclusively on the variables $x_{rs}^i$ with $r < s$. By applying this reduction and substituting the absolute value identity into the objective function, model \eqref{model:MLOP2} can be rewritten in its definitive form:
\begin{subequations}\label{model:MLOP3}
\begin{align}
\max \quad
& \sum_{r=1}^{n-1} \sum_{s=r+1}^n
\Bigl(1-\Bigl|c_{rs}-\sum_{i=1}^g \omega^i x_{rs}^i\Bigr|\Bigr)
\label{obj3}\\
\text{s.t.}\quad
& x_{rs}^i + x_{st}^i - x_{rt}^i \le 1
&& \forall r,s,t\in[[n]],\ i\in[[g]] : r<s<t,
\label{cons:transitivity3_1}\\
& x_{rs}^i + x_{st}^i - x_{rt}^i \ge 0
&& \forall r,s,t\in[[n]],\ i\in[[g]] : r<s<t,
\label{cons:transitivity3_2}\\
& \sum_{i=1}^g \omega^i = 1,
\label{cons:weightsum3}\\
& x_{rs}^i \in \{0,1\}
&& \forall r,s \in [[n]],\ i\in[[g]] : r<s,
\label{cons:binary3}\\
& \omega^i \ge 0
&& \forall i\in[[g]].
\label{cons:wpos3}
\end{align}
\end{subequations}
By ignoring the constant term $\binom{n}{2}$ and changing the sign of the objective, the problem is equivalent to minimizing the $L_1$ distance between the empirical preferences and the weighted mixture of rankings:
\begin{subequations}\label{model:MLOP_L1}
\begin{align}
\min \quad
& \sum_{r=1}^{n-1} \sum_{s=r+1}^n
\Bigl|c_{rs}-\sum_{i=1}^g \omega^i x_{rs}^i\Bigr|
\label{obj_L1}\\
\text{s.t.}\quad
& \eqref{cons:transitivity3_1}, \eqref{cons:transitivity3_2}, \eqref{cons:weightsum3} \nonumber \\
& x_{rs}^i \in \{0,1\}
&& \forall r,s \in [[n]],\ i\in[[g]] : r<s,\\
& \omega^i \ge 0
&& \forall i\in[[g]].
\end{align}
\end{subequations}

We next linearise the nonlinear terms. Introduce variables
$u_{rs}^i=\omega^i x_{rs}^i$ and impose
\begin{subequations}
\begin{align}
u_{rs}^i \le \omega^i
&& \forall r,s\in[[n]],\ \forall i\in[[g]]: r < s, \label{con:u1}\\
u_{rs}^i \le x_{rs}^i
&& \forall r,s\in[[n]],\ \forall i\in[[g]]: r < s,\label{con:u2}\\
u_{rs}^i \ge \omega^i + x_{rs}^i - 1
&& \forall r,s\in[[n]],\ \forall i\in[[g]]: r < s, \label{con:u3}\\
 u_{rs}^i \geq 0
&& \forall r,s\in[[n]],\ \forall i\in[[g]]: r < s. \label{con:u4}
\end{align}
\end{subequations}
Introduce variables $v_{rs}=|c_{rs}-\sum_{i=1}^g u_{rs}^i|$ for $r<s$ and impose
\begin{subequations}
\begin{align}
v_{rs} \ge c_{rs}-\sum_{i=1}^g u_{rs}^i
&& \forall r,s\in[[n]]: r<s, \label{con:v1}\\
v_{rs} \ge \sum_{i=1}^g u_{rs}^i-c_{rs}
&& \forall r,s\in[[n]]: r<s. \label{con:v2}
\end{align}
\end{subequations}

The resulting mixed-integer linear formulation is
\begin{subequations}\label{model:MLOP4}
\begin{align}
\min \quad
& \sum_{r=1}^n \sum_{s=r+1}^n v_{rs}
\label{obj4}\\
\text{s.t.}\quad
& \eqref{cons:transitivity3_1}, \eqref{cons:transitivity3_2}, \eqref{cons:weightsum3}, \eqref{con:u1}, \eqref{con:u2}, \eqref{con:u3}, \eqref{con:u4}, \eqref{con:v1}, \eqref{con:v2} \nonumber\\
& x_{rs}^i \in \{0,1\}
&& \forall r,s\in[[n]],\ \forall i\in[[g]]: r < s,\\
& \omega^i \ge 0
&& \forall i\in[[g]].\\
\end{align}
\end{subequations}

In the following section, we leverage this compact formulation to develop an elegant geometric interpretation of the problem. This perspective allows us to establish a strict theoretical upper bound on the number of groups $g$ required to reach a globally optimal solution. While this geometric analysis provides deep structural insights, it is largely self-contained. Consequently, readers whose primary interest lies in the practical resolution of the model may safely skip to Section~\ref{sec:heuristic}, where we introduce our heuristic approach, without losing the main thread of the paper.

\subsection{Geometric interpretation} \label{subsec:geo}

For each pair $(r,s)$ with $r<s$, the objective function of \eqref{model:MLOP_L1} depends exclusively on the convex combination $\sum_{i=1}^g \omega^i x_{rs}^i$. This structure allows a geometric interpretation of the model. Let $\mathcal L \subseteq \{0,1\}^{\binom{n}{2}}$ denote the set of incidence vectors $\bm x=(x_{rs})_{1\le r<s\le n}$ associated with all valid linear orders of $[[n]]$. The convex hull of this set,
\[
\mathcal P := \operatorname{conv}(\mathcal L),
\]
is the \emph{projected linear ordering polytope}, which is full-dimensional in $\mathbb{R}^{\binom{n}{2}}$ (see, e.g., \cite{marti2022exact}). 

For a fixed number of groups $g \ge 1$, we define the restricted set of points that can be expressed as a convex combination of at most $g$ linear orders:
\[
\mathcal P_g:=\Bigl\{\sum_{i=1}^g \omega^i \bm x^i:\ \bm x^i\in\mathcal L,\
\omega^i\ge0,\ \sum_{i=1}^g\omega^i=1\Bigr\}.
\]
From this viewpoint, solving problem \eqref{model:MLOP_L1} geometrically amounts to finding the closest point to the observed preference data $\bm c$ under the $\ell_1$ norm within $\mathcal P_g$.

\begin{prop}\label{prop:geometry_Pg}
Let $\bm c=(c_{rs})_{1\le r<s\le n}\in\mathbb{R}^{\binom{n}{2}}$ be the vector formed by the upper triangular entries of the preference matrix $C$. For any $g \ge 1$, the optimal value of \eqref{model:MLOP_L1} satisfies
\[
\mathrm{OPT}_g = \min_{\bm p\in\mathcal P_g}\|\bm p-\bm c\|_1.
\]
\end{prop}

\begin{proof}
Consider a feasible solution $(\bm x,\bm\omega)$ of \eqref{model:MLOP_L1}. For each $i\in[[g]]$, the vector $\bm x^i=(x_{rs}^i)_{1\le r<s\le n}$ corresponds to a linear order and therefore belongs to $\mathcal L$. Hence, the convex combination $\bm p:=\sum_{i=1}^g \omega^i \bm x^i$ lies in $\mathcal P_g$. Conversely, every $\bm p\in\mathcal P_g$ arises from some feasible assignment $(\bm x,\bm\omega)$. Therefore, minimizing the objective function \eqref{obj3} over all feasible solutions is equivalent to minimizing the $\ell_1$-distance between $\bm c$ and the elements of $\mathcal P_g$.
\end{proof}

This geometric perspective is particularly useful for theoretically bounding the maximum number of groups needed to explore the entire polytope $\mathcal P$. Carath\'eodory's theorem (originally introduced by \cite{caratheodory1911variabilitatsbereich}; see also \cite[Theorem 17.1]{rockafellar1970convex}) states that any point in the convex hull of a set in a $d$-dimensional space can be expressed as a convex combination of at most $d+1$ points from that set. Since $\mathcal P$ lies in a space of dimension $d = \binom{n}{2}$, Carath\'eodory's theorem implies that $\mathcal P_g = \mathcal P$ for any $g \ge \binom{n}{2} + 1$. This naturally leads to the following theoretical bound:

\begin{coro}\label{cor:geometry_full}
For any $g\ge\binom{n}{2}+1$, the optimal value of \eqref{model:MLOP_L1} satisfies
\[
\mathrm{OPT}_g = \min_{\bm p\in\mathcal P}\|\bm p-\bm c\|_1.
\]
\end{coro}

We next illustrate the geometric interpretation of Proposition~\ref{prop:geometry_Pg} with a simple instance involving three alternatives.

\begin{example} \label{ej:geom_1}

Consider the normalized preference matrix
\[
C=
\begin{pmatrix}
- & 0.7 & 0.8\\
0.3 & - & 0.4\\
0.2 & 0.6 & -
\end{pmatrix},
\qquad 
\mathbf{c}=(c_{12},c_{13},c_{23})=(0.7,0.8,0.4)\in\mathbb{R}^3,
\]
where $\mathbf{c}$ collects the upper–triangular entries.

For $n=3$, each linear order induces a binary vector
\[
\mathbf{x}=(x_{12},x_{13},x_{23})\in\{0,1\}^3,
\qquad x_{rs}=1 \text{ if } r\succ s,
\]
so that the polytope $\mathcal P=\mathrm{conv}(\mathcal L)$ is the convex hull
of the six vertices reported in Table~\ref{tab:orders}. Geometrically,
$\mathcal P$ is a three–dimensional polytope encoding all strict rankings,
whereas $\mathbf{c}$ represents empirical pairwise intensities and typically
lies in its interior.

For $n=3$, $\mathcal P$ also admits the compact representation
\[
\mathcal P=\{(x,y,z)\in\mathbb{R}^3:\ 0\le x,y,z\le 1,\ 0\le x-y+z\le 1\},
\]
where $(x,y,z)=(x_{12},x_{13},x_{23})$.
The additional inequalities $0\le x-y+z\le 1$ encode the transitivity
constraints for the unique triple $\{1,2,3\}$.

\begin{table}[htp!]
\caption{Linear orders and their associated vertices for $n=3$.}
\centering
\small
\setlength{\tabcolsep}{6pt}
\renewcommand{\arraystretch}{1.15}
\begin{tabular}{c|cccccc}
\hline
Linear order 
& $1\succ2\succ3$ & $1\succ3\succ2$ & $2\succ1\succ3$ & $2\succ3\succ1$ & $3\succ1\succ2$ & $3\succ2\succ1$ \\
\hline
Vertex in $\mathbb{R}^3$
& $(1,1,1)$ & $(1,1,0)$ & $(0,1,1)$ & $(0,0,1)$ & $(1,0,0)$ & $(0,0,0)$ \\
\hline
\end{tabular}
\label{tab:orders}
\end{table}

Model~\eqref{model:MLOP_L1} searches for a convex combination of at most $g$
vertices of $\mathcal P$ approximating $\mathbf{c}$. As $g$ increases,
the admissible set enlarges from individual vertices to higher–dimensional
faces of the polytope. Any feasible solution $(\bm x,\bm \omega)$ induces
a point
\[
\bm p =
\left(
\sum_{i=1}^g \omega^i x_{12}^i,\;
\sum_{i=1}^g \omega^i x_{13}^i,\;
\sum_{i=1}^g \omega^i x_{23}^i
\right),
\]
which corresponds to the convex combination of the vertices associated
with the selected rankings. In this instance, the optimal solutions are
\begin{align*}
\bm p_1 &= (1,1,0),\\
\bm p_2 &= 0.7\, (1,1,0)+0.3\, (0,0,1)=(0.7,0.7,0.3),\\
\bm p_3 &= 0.4\, (1,1,1)+0.4\, (1,1,0)+0.2\, (0,0,0)=(0.8,0.8,0.4),\\
\bm p_4 &= 0.6\, (1,1,0)+0.2\, (0,0,1)+0.1\, (0,1,1)+0.1\, (1,1,1)
=(0.7,0.8,0.4)=\mathbf{c}.
\end{align*}

Thus, as illustrated in Figure~\ref{fig:MLOP-geometry}, the solution moves
from a vertex ($g=1$) to a convex combination of two vertices ($g=2$),
which lies on the segment joining them (not necessarily on an edge of $\mathcal P$).
For $g=3$, the solution belongs to the convex hull of three vertices; in this case,
the optimal solution is not unique and distinct points attain the same optimal value.
Figure~\ref{fig:MLOP-geometry} displays one such optimal solution.
Finally, for $g=4$ the model reaches the interior point $\mathbf{c}$,
exactly reproducing the data. Since $\mathbf{c}$ lies in the interior of $\mathcal P$,
its representation as a convex combination of vertices is not unique, and multiple optimal solutions arise.

\begin{figure}
\centering
\caption{Geometric evolution of the optimal solutions $\bm p_g$ in Example~\ref{ej:geom_1} as $g$ increases. 
The point $\bm c$ is shown in orange. In each panel, the vertices used in the corresponding optimal convex combination and the optimal point $\bm p_g$ are shown in purple, with $\bm p_g$ explicitly highlighted.}
\begin{subfigure}{0.48\textwidth}
    \centering
    \includegraphics[width=0.9\textwidth]{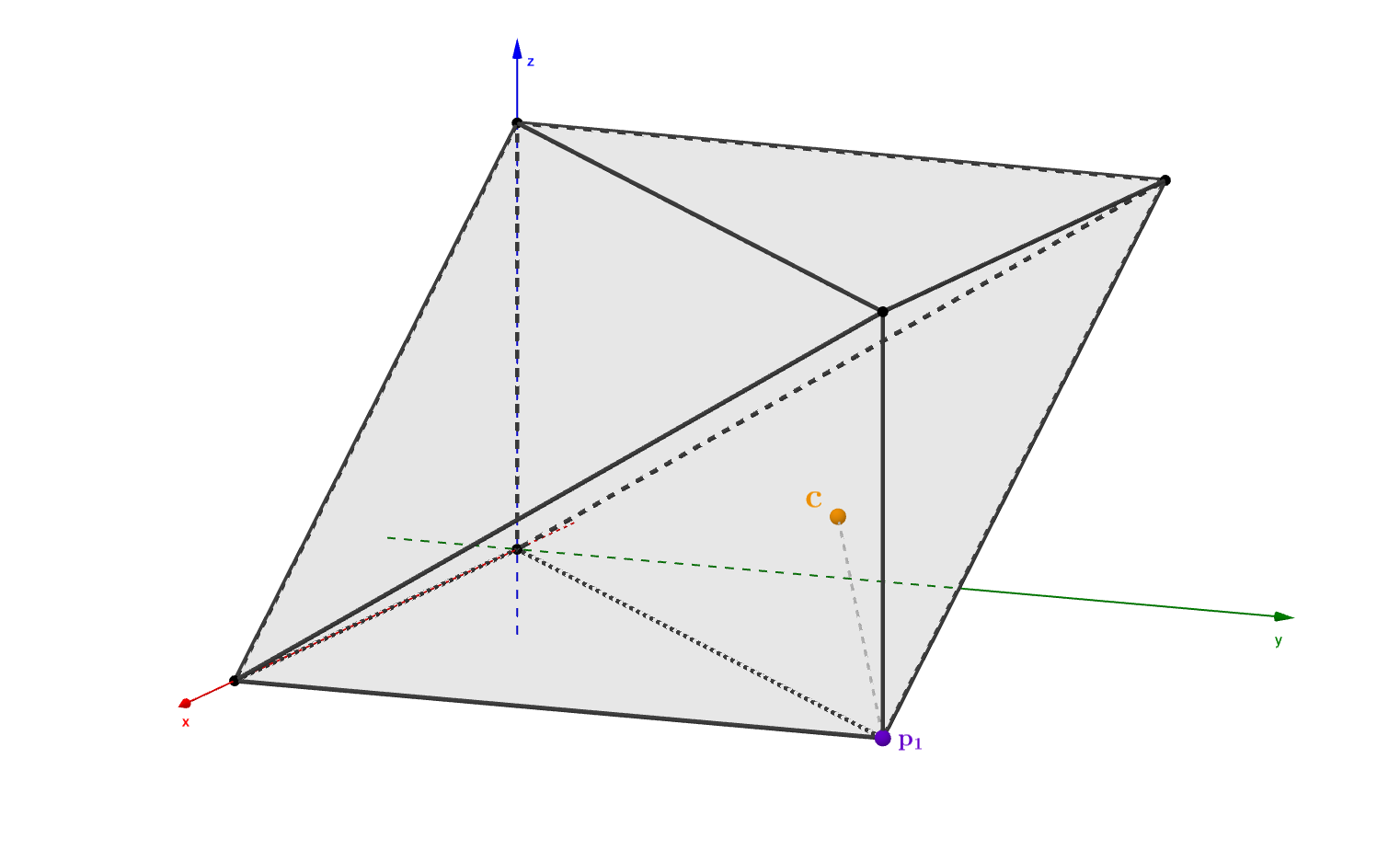}
    \caption{$g=1$}
\end{subfigure}
\hfill
\begin{subfigure}{0.48\textwidth}
    \centering
    \includegraphics[width=0.9\textwidth]{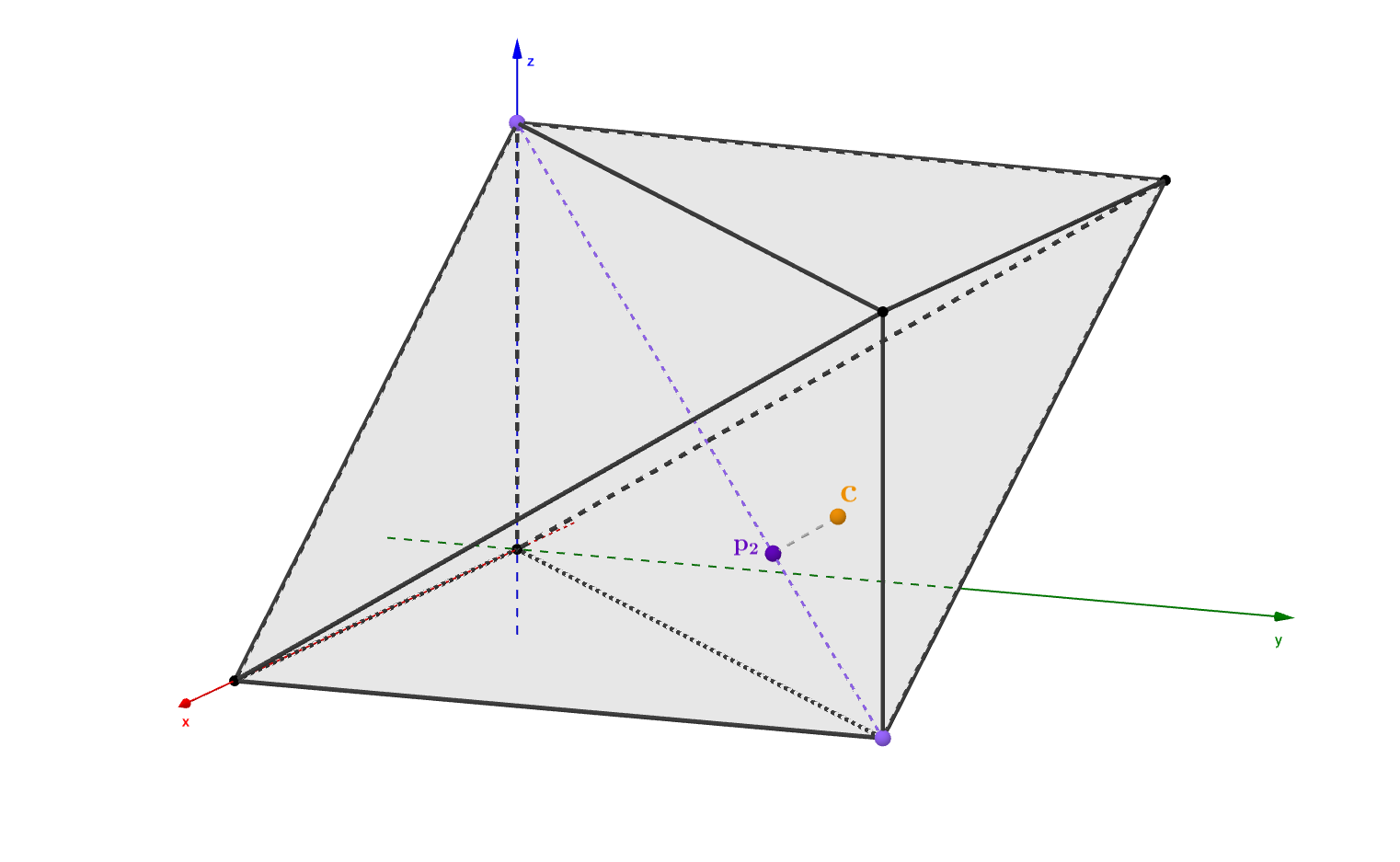}
    \caption{$g=2$}
\end{subfigure}

\medskip

\begin{subfigure}{0.48\textwidth}
    \centering
    \includegraphics[width=0.9\textwidth]{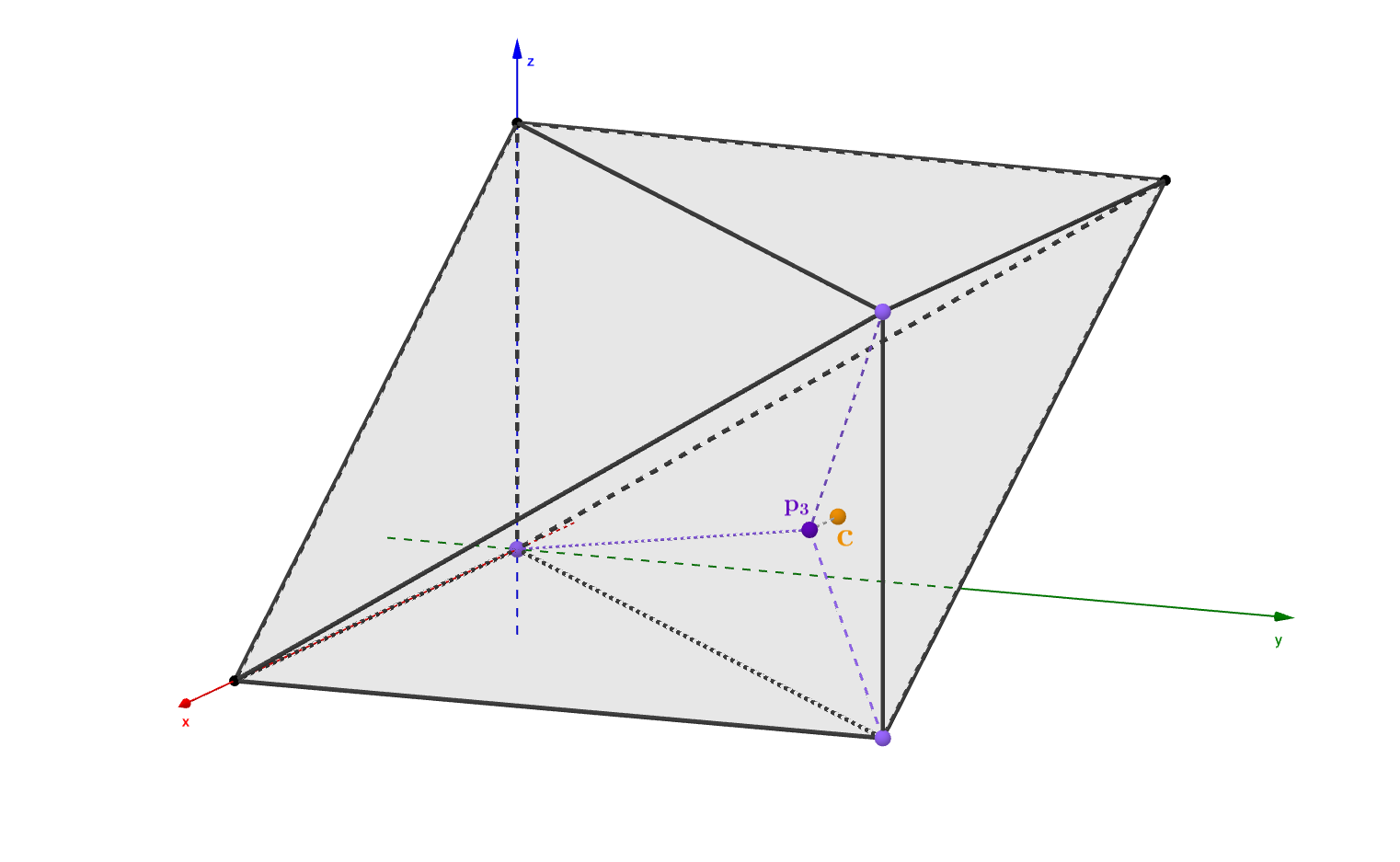}
    \caption{$g=3$}
\end{subfigure}
\hfill
\begin{subfigure}{0.48\textwidth}
    \centering
    \includegraphics[width=0.9\textwidth]{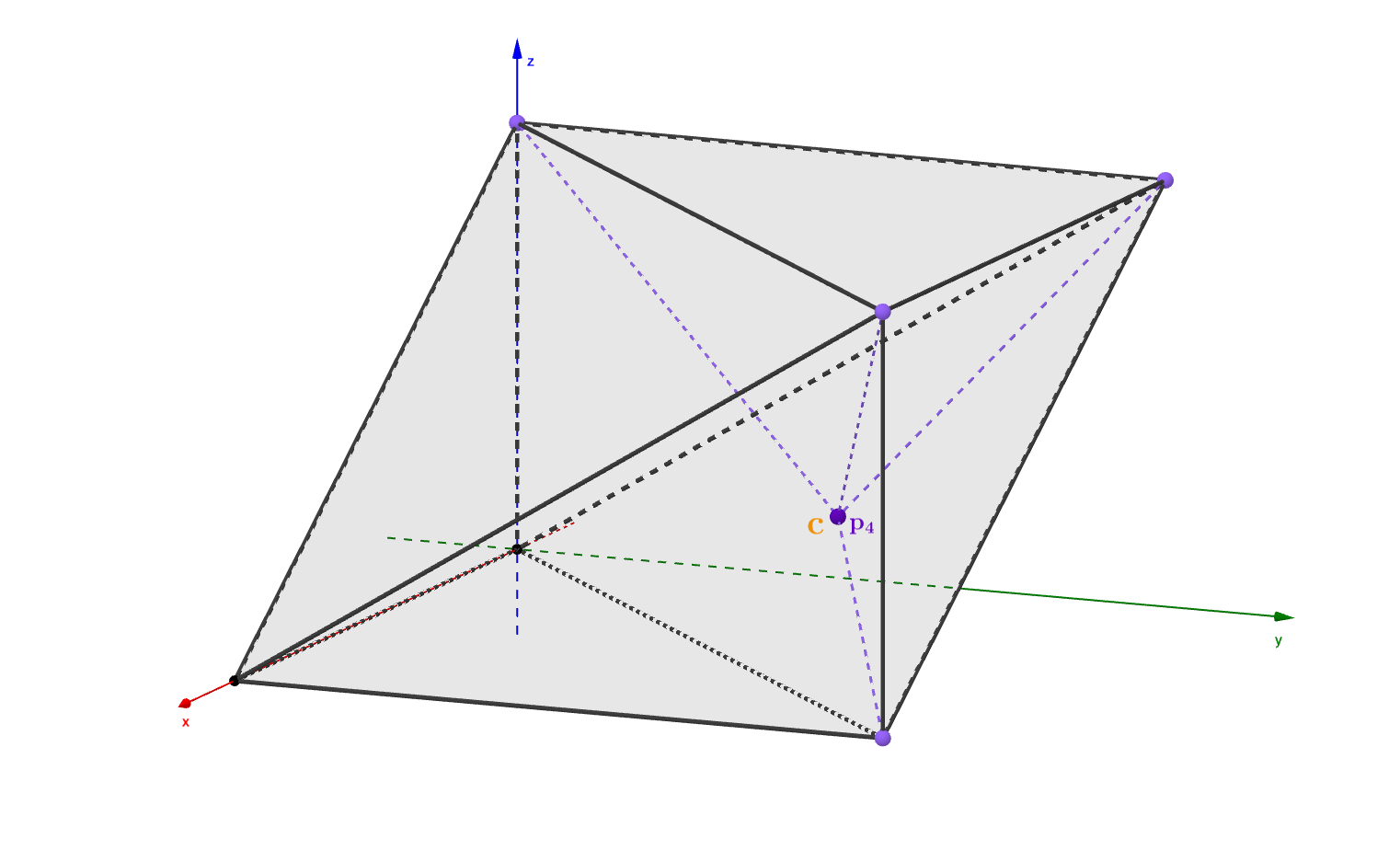}
    \caption{$g=4$}
\end{subfigure}
\label{fig:MLOP-geometry}
\end{figure}

\end{example}

If $\bm c$ does not belong to the interior of $\mathcal P$ (which occurs, for instance, when the empirical data violates a transitivity constraint), its projection lies on a proper face of the polytope. Consequently, Carath\'eodory's bound can be tightened by one.

\begin{coro}\label{cor:geometry_outside}
If $\bm c \notin \operatorname{int}\mathcal P$, then for any
$g\ge \binom{n}{2}$ the optimal value of \eqref{model:MLOP_L1} satisfies
\[
\mathrm{OPT}_g=\min_{\bm p\in\mathcal P}\|\bm p-\bm c\|_1.
\]
\end{coro}

\begin{proof}
Let $\bm c^*$ be an $\ell_1$-projection of $\bm c$ onto $\mathcal P$.
Since $\bm c\notin\operatorname{int}\mathcal P$ and $\mathcal P$ is a polytope,
any closest point belongs to the boundary $\partial\mathcal P$ and therefore lies
in a proper face $F$ of $\mathcal P$. Hence $\dim(F)\le \binom{n}{2}-1$.

Applying Carath\'eodory's theorem in the affine hull of $F$, the point
$\bm c^*$ can be written as a convex combination of at most
$\binom{n}{2}$ vertices of $F$, which are also vertices of $\mathcal P$
and thus belong to $\mathcal L$. Consequently,
$\bm c^*\in\mathcal P_g$ for every $g\ge \binom{n}{2}$.
The result follows.
\end{proof}

We next examine a situation in which the empirical vector lies outside the linear ordering polytope. In this case, the $\ell_1$-projection of $\bm c$ onto $\mathcal P$ necessarily belongs to the boundary of the polytope and, more precisely, to a proper face. As a consequence, the number of vertices required to represent an optimal solution can be reduced by one, in accordance with Corollary~\ref{cor:geometry_outside}.

\begin{example} \label{ej:geom_2}

Consider the normalized preference matrix
\[
C=
\begin{pmatrix}
- & 0.3 & 0.9\\
0.7 & - & 0.2\\
0.1 & 0.8 & -
\end{pmatrix},
\qquad 
\mathbf{c}=(c_{12},c_{13},c_{23})=(0.3,0.9,0.2)\in\mathbb{R}^3.
\]

For $n=3$, the linear ordering polytope admits the description
\[
\mathcal P=\{(x,y,z)\in\mathbb{R}^3:\ 0\le x,y,z\le 1,\ 0\le x-y+z\le 1\}.
\]
Evaluating the transitivity expression at $\mathbf c$ yields
\[
0.3-0.9+0.2=-0.4<0,
\]
so $\mathbf c\notin\mathcal P$. Hence the model seeks the closest point of $\mathcal P$ to $\mathbf c$ in the $\ell_1$-norm.

Model~\eqref{model:MLOP_L1} searches for a convex combination of at most $g$
vertices of $\mathcal P$ minimizing the $\ell_1$-distance to $\mathbf c$.
The points $p_g$ are constructed from an optimal solution $(\bm x,\bm \omega)$
as in Example~\ref{ej:geom_1}. The optimal solutions are given by
\begin{align*}
\bm p_1 &= (1,1,0),\\
\bm p_2 &= 0.5\, (1,1,0)+0.5\, (0,1,1)=(0.5,1,0.5),\\
\bm p_3 &= 0.4\, (0,0,0)+0.3\, (0,1,1)+0.3\, (1,1,0)=(0.3,0.6,0.3).
\end{align*}

For $g=1$, the optimal solution coincides with a single vertex of $\mathcal P$, whereas for $g=2$ it lies on a segment joining two vertices. When $g=3$, the model attains the $\ell_1$-projection of $\mathbf c$ onto $\mathcal P$, which belongs to a proper face of the polytope. In this case, the optimal solution is not unique, and alternative convex combinations of three vertices achieve the same optimal value. No additional groups are required, consistently with Corollary~\ref{cor:geometry_outside}. Figure~\ref{fig:MLOP-geometry-outside} illustrates this geometric behavior.

\begin{figure}
\centering
\caption{Geometric evolution of the optimal solutions $\bm p_g$ in Example~\ref{ej:geom_2} as $g$ increases. 
The point $\bm c$ is shown in orange. In each panel, the vertices used in the corresponding optimal convex combination and the optimal point $\bm p_g$ are shown in purple, with $\bm p_g$ explicitly highlighted.}
\begin{subfigure}{0.48\textwidth}
    \centering
    \includegraphics[width=\textwidth]{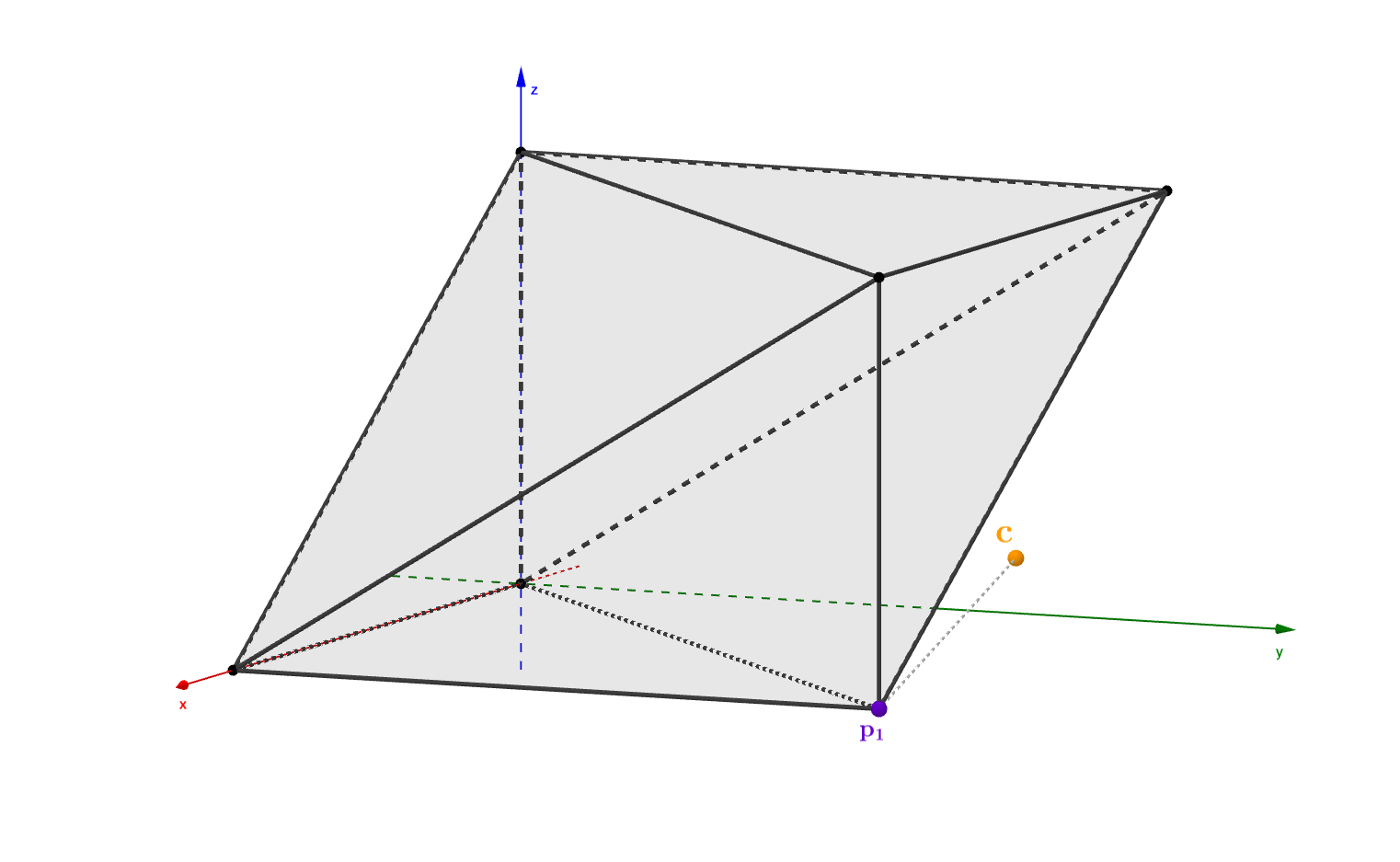}
    \caption{$g=1$}
\end{subfigure}
\hfill
\begin{subfigure}{0.48\textwidth}
    \centering
    \includegraphics[width=\textwidth]{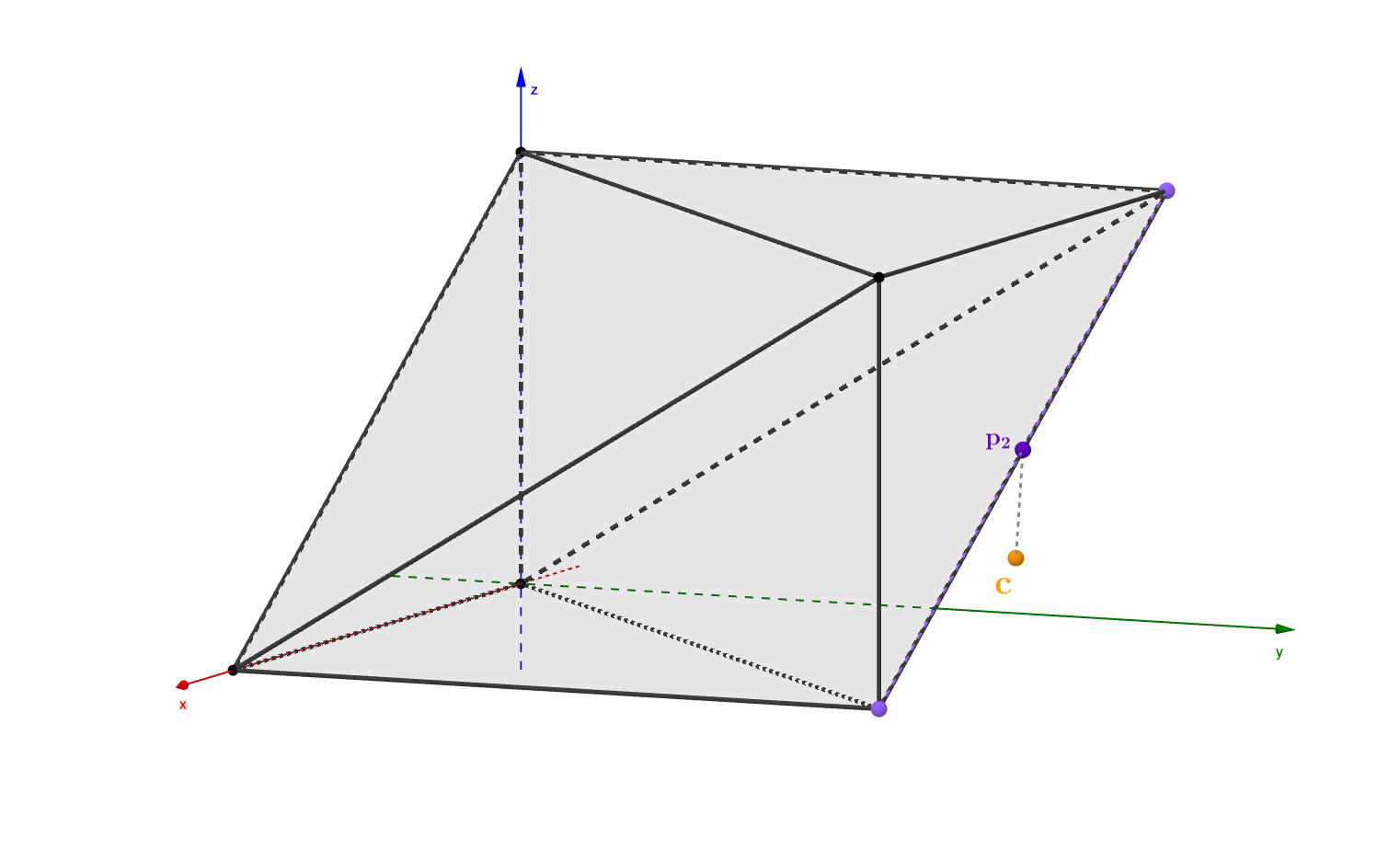}
    \caption{$g=2$}
\end{subfigure}

\medskip

\begin{subfigure}{0.48\textwidth}
    \centering
    \includegraphics[width=\textwidth]{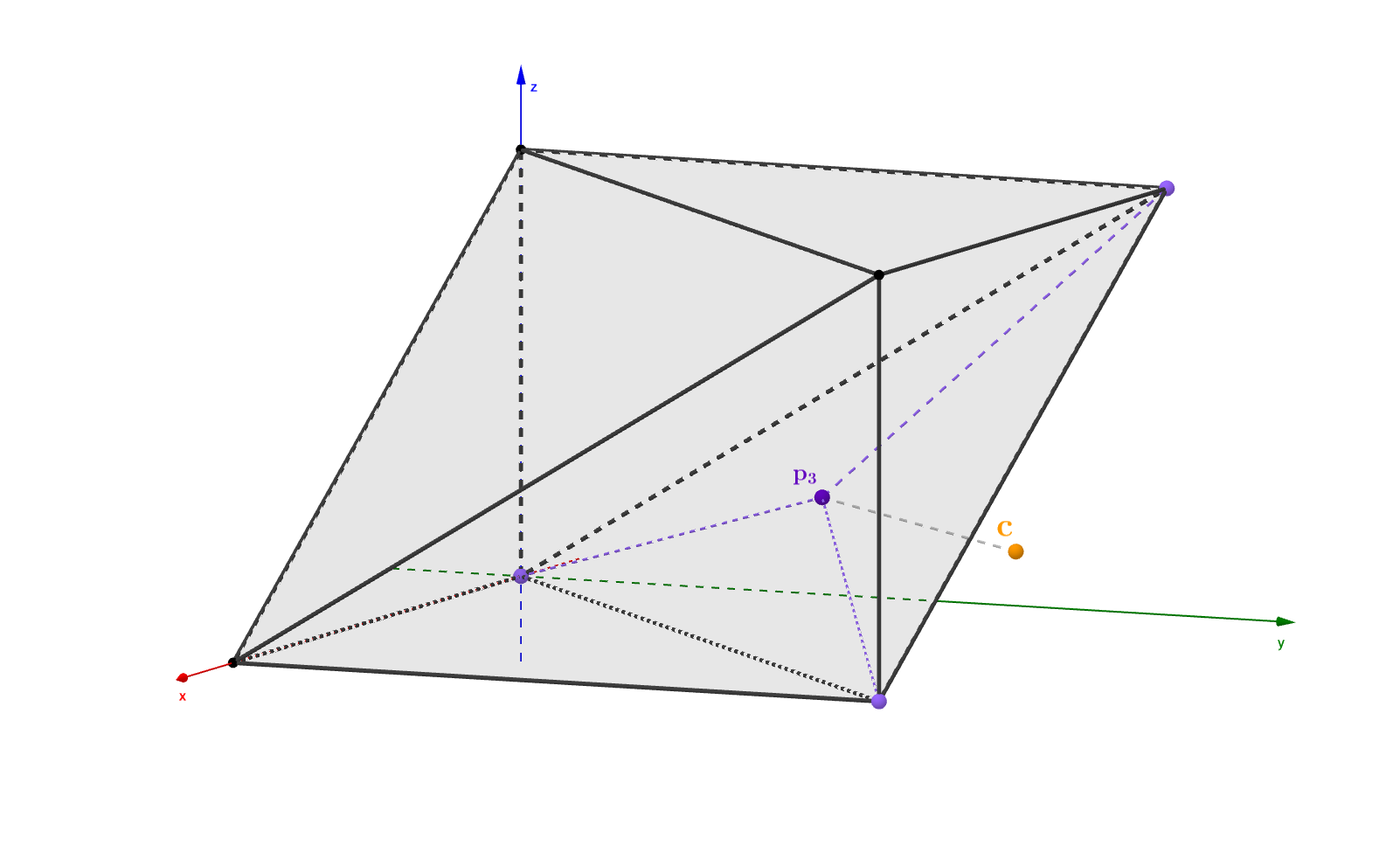}
    \caption{$g=3$}
\end{subfigure}
\label{fig:MLOP-geometry-outside}
\end{figure}

\end{example}

\section{Heuristic solution via Alternating Direction Methods} \label{sec:heuristic}

Solving ${\rm MLOP}_g$ exactly is computationally challenging. Even the classical Linear Ordering Problem (LOP), which corresponds to the case $g = 1$, is NP-hard \citep{garey1979computers}, as it requires optimizing a single linear order over $n$ items. The extension considered in this work further increases the complexity of the problem, since up to $g$ rankings must be determined simultaneously, alongside their associated population weights and latent preference structures. Consequently, computing optimal solutions to ${\rm MLOP}_g$ becomes intractable for all but the smallest instances, motivating the development of heuristic approaches capable of producing high-quality solutions within reasonable computational timeframes.

The solution approach developed in this section is an alternating-direction heuristic. More precisely, it can be characterized as a \textit{matheuristic}, since each iteration solves optimization subproblems over a subset of decision variables while keeping the remaining variables fixed. This approach follows the general philosophy of alternating-direction methods commonly used in mixed-integer optimization and is inspired, in particular, by the overlapping alternating direction method proposed by \cite{cattaruzza2024exact}. While the structure of 
$\mathrm{MLOP}_g$ differs from the quantile minimisation setting studied in 
their work, the underlying idea remains the same: instead of optimising all 
decision variables simultaneously, the algorithm alternates between simpler 
subproblems in which only a subset of variables is updated while the others are 
kept fixed. This alternating decomposition allows the method to exploit the 
structure of the model and substantially reduces the computational effort 
required at each iteration.

In the compact formulation of $\mathrm{MLOP}_g$, each group is described by a 
ranking $\sigma^i$ and an associated weight $\omega^i$, while the latent 
preference structure is implicitly represented through the convex combination 
induced by these quantities. This perspective naturally leads to a decomposition 
of the problem into two complementary optimisation directions. 

In the first direction, the population weights are kept fixed and the group 
rankings are updated so as to better approximate the observed preference matrix. 
In the second direction, the rankings are fixed and the population weights are 
re-optimised to obtain the best convex combination of the current orders. 
By alternating between these two steps, the heuristic progressively improves 
the agreement between the set of rankings, their mixture weights, and the 
observed preferences. Crucially, this alternating scheme completely avoids the need 
for the auxiliary continuous variables $u_{rs}^i$ introduced in the compact formulation 
to linearise the product $\omega^i x_{rs}^i$. Since either $\bm{\omega}$ or $\bm{x}$ 
is treated as a fixed parameter at each step, their product is inherently linear, 
drastically reducing the number of variables and constraints in the resulting subproblems.

Regarding the initialisation of the alternating procedure, we generate starting points by drawing random vectors over the probability simplex for the group weights. To ensure robustness and reduce sensitivity to the starting point, the heuristic can be executed from multiple random initialisations, and the overall best solution obtained is retained.  

We first update the group rankings while keeping the weight vector fixed. 
Given $\hat{\bm{\omega}}=(\hat{\omega}^1,\dots,\hat{\omega}^g)$, the ranking-update 
step solves the following mixed-integer linear program:
\begin{subequations}\label{model:Step1}
\begin{align}
\min \quad & \sum_{r=1}^n \sum_{s=r+1}^n v_{rs}
\label{eq:step1_obj}\\
\text{s.t.}\quad
& x_{rs}^i + x_{st}^i - x_{rt}^i \le 1 && \forall\, r,s,t \in [[n]],\, \forall\, i \in [[g]]: r < s < t,
\label{eq:step1_trans1}\\
& x_{rs}^i + x_{st}^i - x_{rt}^i \ge 0 && \forall\, r,s,t \in [[n]],\, \forall\, i \in [[g]]: r < s < t,
\label{eq:step1_trans2}\\
& v_{rs} \ge c_{rs}-\sum_{i=1}^g \hat{\omega}^i x_{rs}^i
&& \forall\, r,s \in [[n]]: r<s,
\label{eq:step1_abs1}\\
& v_{rs} \ge \sum_{i=1}^g \hat{\omega}^i x_{rs}^i - c_{rs}
&& \forall\, r,s \in [[n]]: r<s,
\label{eq:step1_abs2}\\
& x_{rs}^i \in \{0,1\}
&& \forall\, r,s \in [[n]],\, \forall\, i \in [[g]]: r < s.
\label{eq:step1_bin}
\end{align}
\end{subequations}

Since \eqref{model:Step1} is a mixed-integer linear program, solving it to optimality 
at each iteration can be time-consuming. However, as our framework is a heuristic, 
it is not strictly necessary to find the exact optimal group rankings at every step to achieve 
convergence. Therefore, to significantly accelerate the process, a strict time limit can be imposed 
on the solver during this phase. If the optimum is not proven within this limit, the method simply 
extracts the best incumbent solution found so far and uses it as the outcome of the step.

Solving~\eqref{model:Step1} yields an updated set of rankings 
$\hat{\bm{x}}$. In the second optimisation direction, the rankings 
are kept fixed and the population weights are re-optimised so as to 
obtain the best convex combination of the current orders. Since all integer variables 
are fixed, this reduces to a simple linear program:
\begin{subequations} \label{model:Step2}
\begin{align} 
\min \quad & \sum_{r=1}^n \sum_{s=r+1}^n v_{rs} 
\label{eq:step2_obj}\\
\text{s.t.}\quad
    & \sum_{i=1}^g \omega^i = 1
    \label{eq:step2_simplex}\\
    & v_{rs} \ge c_{rs}-\sum_{i=1}^g \omega^i \hat x_{rs}^i
    && \forall\, r,s \in [[n]]: r<s,
    \label{eq:step2_abs1}\\
    & v_{rs} \ge \sum_{i=1}^g \omega^i \hat x_{rs}^i - c_{rs}
    && \forall\, r,s \in [[n]]: r<s,
    \label{eq:step2_abs2}\\
    & \omega^i \ge 0
    && \forall\, i \in [[g]].
    \label{eq:step2_nonneg_w}
\end{align}
\end{subequations}
Solving~\eqref{model:Step2} yields an updated weight vector 
$\hat{\bm{\omega}}$, which becomes the reference point for the next 
iteration of the alternating scheme.

The two optimisation phases described above are combined into a general iterative alternating-direction heuristic, which is summarised in Algorithm~\ref{alg:heuristic}.
\begin{algorithm}[htp!!]
\footnotesize
\caption{Alternating-direction heuristic (Multi-start)}\label{alg:heuristic}
\DontPrintSemicolon
\LinesNotNumbered

\SetKwInput{KwData}{Input}
\SetKwInput{KwResult}{Output}

\KwData{%
\begin{tabular}[t]{@{}p{13.2cm}@{}}
Number of items $n$, Maximum number of groups $g$ \;
Normalised preference matrix $C = (c_{rs})_{n \times n}$ \;
Maximum number of iterations $It_{max}$ \;
Tolerance $\varepsilon$ \;
Number of starts $N_{start}$ \;
Max time limit for MILP phase $T_{max}$
\end{tabular}
}
\KwResult{%
\begin{tabular}[t]{@{}p{13.2cm}@{}}
Best heuristic solution $(\hat{\bm{x}}, \hat{\bm{\omega}})$
\end{tabular}
}

Initialise $Obj_{best} \gets +\infty$ \textbf{and} $k \gets 1$ \;

\While{$k \le N_{start}$}{

  Generate random feasible weight vector $\hat{\bm{\omega}}$ on the probability simplex \;
  
  Set $I \gets 1$ \textbf{and} $Obj_{local} \gets +\infty$ \;

  \While{$I \le It_{max}$}{

    $Obj_{start} \gets Obj_{local}$ \tcp*[l]{Objective at the start of the iteration}

    Solve model~\eqref{model:Step1} with fixed $\hat{\bm{\omega}}$ (time limit $T_{max}$) yielding $\bm{x}$ and $Obj_1$ \;
    \If{$Obj_1 < Obj_{local}$}{
      $\tilde{\bm{x}} \gets \bm{x}$ \textbf{and} $Obj_{local} \gets Obj_1$ \;
    }

    Solve model~\eqref{model:Step2} with fixed $\hat{\bm{x}}$ yielding $\bm{\omega}$ and $Obj_2$ \;
    \If{$Obj_2 < Obj_{local}$}{
      $\tilde{\bm{\omega}} \gets \bm{\omega}$ \textbf{and} $Obj_{local} \gets Obj_2$ \;
    }

    \If{$|Obj_{start} - Obj_{local}| < \varepsilon$}{
      \textbf{break} \tcp*[l]{Stop if no significant improvement in this iteration}
    }

    $I \gets I + 1$ \;
  }

  \If{$Obj_{local} < Obj_{best}$}{
    $Obj_{best} \gets Obj_{local}$ \;
    $(\hat{\bm{x}},\hat{\bm{\omega}}) \gets (\tilde{\bm{x}}, \tilde{\bm{\omega}})$ \;
  }
  
  $k \gets k + 1$ \;
}

\Return{$(\hat{\bm{x}}, \hat{\bm{\omega}})$}\;

\end{algorithm}

To ensure computational efficiency, the execution logic of the algorithm is governed by several interacting parameters. The heuristic runs for a predetermined number of independent random starts ($N_{start}$), which dictates the outer loop of the algorithm. Within each start, the alternating scheme terminates either when the improvement in the objective value between consecutive iterations falls below a strict tolerance $\varepsilon$, or when a maximum number of local iterations ($It_{max}$) is reached. Finally, to prevent bottlenecks during the most demanding phase, the computational effort spent on the MILP subproblem~\eqref{model:Step1} is strictly bounded by the time limit $T_{max}$.

\section{Computational experiment} \label{sec:exp_com}

This section evaluates the computational performance of the proposed models and algorithms. \footnote{All instances, model files, and computational results are available in the supplementary repository at \url{https://doi.org/10.5281/zenodo.19709333}.}
The experiments pursue two main objectives. First, we analyse the behaviour of the exact 
approach, relying exclusively on the compact mixed-integer formulation \eqref{model:MLOP4}, on a collection of 
synthetic instances in order to assess its ability to recover the latent preference structures 
used to generate the data, including the underlying rankings and the associated group weights. 
Second, we compare this exact approach with the proposed heuristic in terms of solution quality 
and computational time, evaluating their scalability as the problem size and the number of 
groups increase.

All tests were performed using the commercial IP solver \emph{Gurobi} on a server equipped with an AMD Ryzen~9~7950X processor (3.4~GHz, 16~cores and 32~threads) and 4\,$\times$\,48~GB of DDR5 RAM (196~GB in total). The system is configured to allow up to 16 single-threaded tasks to run concurrently, each with approximately 16~GB of allocated memory, and relies on 2\,$\times$\,2~TB Western Digital Black SN770 NVMe SSDs in RAID~1 to ensure data reliability during execution. In our experiments, each job was assigned 8 computational threads and executed using the \emph{Gurobi} solver through Python implementations.

\subsection{Synthetic instance generation}

Since clean publicly available datasets with known heterogeneous preference structures are difficult to obtain, we generate synthetic instances in a controlled manner to validate our approach. Each instance is defined by the number of items $n$, the number of latent groups $\overline{g}$, the group weights $\omega^i$, and a dispersion parameter $D$, defined as
\[
D=\left\lceil p\binom{n}{2}\right\rceil,
\]
where $p$ is the percentage of the maximum Kendall distance $\binom{n}{2}$.

Instances are generated according to the following procedure. First, $\overline{g}$ central permutations of the $n$ items are generated to represent the reference ranking of each group. These permutations are sampled at random while enforcing a minimum Kendall distance between any pair of them, ensuring that the reference rankings are sufficiently distinct. Next, rankings are generated around the central permutation of each group with Kendall distance at most $D$. A total of 1000 rankings are generated and distributed among the groups proportionally to their weights $\omega^i$. Finally, all generated rankings are aggregated to construct the preference matrix $C$.

In the computational experiments, we consider instances with $n=12$ and $n=24$. For each value of $n$, we generate instances with $\overline{g}=2$, $\overline{g}=3$, and $\overline{g}=4$ latent groups. For each of these cases, we consider three noise levels, namely $p=1$, $p=5$, and $p=10$, corresponding to 1\%, 5\%, and 10\% of the maximum Kendall distance, respectively. Finally, for every combination of $n$, $\overline{g}$, and $p$, we consider two different weight structures: one with equal group weights, and another one in which each group has twice the weight of the following group. Altogether, this yields a total of 36 synthetic instances. The characteristics of all instances are summarized in Table~\ref{tab:instance_parameters}.

\begin{table}[htp!]
\centering
\scriptsize
\caption{Characteristics of the synthetic instances.}
\label{tab:instance_parameters}

\begin{minipage}[t]{0.49\textwidth}
\centering
\begin{tabular}{c c c c c l}
\toprule
Instance & $n$ & $\overline{g}$ & $p$ & $D$ & $\overline{\bm{\omega}}$ \\
\midrule
$R_1$  & 12 & 2 & 1  & 1 & $(0.667,0.333)$ \\
$R_2$  & 12 & 2 & 1  & 1 & $(0.500,0.500)$ \\
$R_3$  & 12 & 2 & 5  & 4 & $(0.667,0.333)$ \\
$R_4$  & 12 & 2 & 5  & 4 & $(0.500,0.500)$ \\
$R_5$  & 12 & 2 & 10 & 8 & $(0.667,0.333)$ \\
$R_6$  & 12 & 2 & 10 & 8 & $(0.500,0.500)$ \\[0.3em]

$R_7$  & 12 & 3 & 1  & 1 & $(0.571,0.286,0.143)$ \\
$R_8$  & 12 & 3 & 1  & 1 & $(0.334,0.333,0.333)$ \\
$R_9$  & 12 & 3 & 5  & 4 & $(0.571,0.286,0.143)$ \\
$R_{10}$ & 12 & 3 & 5  & 4 & $(0.334,0.333,0.333)$ \\
$R_{11}$ & 12 & 3 & 10 & 8 & $(0.571,0.286,0.143)$ \\
$R_{12}$ & 12 & 3 & 10 & 8 & $(0.334,0.333,0.333)$ \\[0.3em]

$R_{13}$ & 12 & 4 & 1  & 1 & $(0.533,0.257,0.133,0.067)$ \\
$R_{14}$ & 12 & 4 & 1  & 1 & $(0.250,0.250,0.250,0.250)$ \\
$R_{15}$ & 12 & 4 & 5  & 4 & $(0.533,0.257,0.133,0.067)$ \\
$R_{16}$ & 12 & 4 & 5  & 4 & $(0.250,0.250,0.250,0.250)$ \\
$R_{17}$ & 12 & 4 & 10 & 8 & $(0.533,0.257,0.133,0.067)$ \\
$R_{18}$ & 12 & 4 & 10 & 8 & $(0.250,0.250,0.250,0.250)$ \\
\bottomrule
\end{tabular}
\end{minipage}
\hfill
\begin{minipage}[t]{0.49\textwidth}
\centering
\begin{tabular}{c c c c c l}
\toprule
Instance & $n$ & $\overline{g}$ & $p$ & $D$ & $\overline{\bm{\omega}}$ \\
\midrule
$R_{19}$ & 24 & 2 & 1  & 1 & $(0.667,0.333)$ \\
$R_{20}$ & 24 & 2 & 1  & 1 & $(0.500,0.500)$ \\
$R_{21}$ & 24 & 2 & 5  & 4 & $(0.667,0.333)$ \\
$R_{22}$ & 24 & 2 & 5  & 4 & $(0.500,0.500)$ \\
$R_{23}$ & 24 & 2 & 10 & 8 & $(0.667,0.333)$ \\
$R_{24}$ & 24 & 2 & 10 & 8 & $(0.500,0.500)$ \\[0.3em]

$R_{25}$ & 24 & 3 & 1  & 1 & $(0.571,0.286,0.143)$ \\
$R_{26}$ & 24 & 3 & 1  & 1 & $(0.334,0.333,0.333)$ \\
$R_{27}$ & 24 & 3 & 5  & 4 & $(0.571,0.286,0.143)$ \\
$R_{28}$ & 24 & 3 & 5  & 4 & $(0.334,0.333,0.333)$ \\
$R_{29}$ & 24 & 3 & 10 & 8 & $(0.571,0.286,0.143)$ \\
$R_{30}$ & 24 & 3 & 10 & 8 & $(0.334,0.333,0.333)$ \\[0.3em]

$R_{31}$ & 24 & 4 & 1  & 1 & $(0.533,0.257,0.133,0.067)$ \\
$R_{32}$ & 24 & 4 & 1  & 1 & $(0.250,0.250,0.250,0.250)$ \\
$R_{33}$ & 24 & 4 & 5  & 4 & $(0.533,0.257,0.133,0.067)$ \\
$R_{34}$ & 24 & 4 & 5  & 4 & $(0.250,0.250,0.250,0.250)$ \\
$R_{35}$ & 24 & 4 & 10 & 8 & $(0.533,0.257,0.133,0.067)$ \\
$R_{36}$ & 24 & 4 & 10 & 8 & $(0.250,0.250,0.250,0.250)$ \\
\bottomrule
\end{tabular}
\end{minipage}

\end{table}

\subsection{Exact compact model performance}

This section evaluates the computational tractability of the exact mixed-integer formulation \eqref{model:MLOP4} on synthetic instances and analyzes how the representation quality evolves as the number of groups ($g$) increases. 

To evaluate and compare solution quality across different values of $g$ and instance sizes, we define a normalised measure of fit inspired by the \emph{degree of linearity} of the classical LOP \cite{marti2022exact}:
\[
\mathrm{Fit}(\bm x, \bm \omega) \;=\; 1-\frac{\sum_{r=1}^{n-1} \sum_{s=r+1}^n
\Bigl|c_{rs}-\sum_{i=1}^g \omega^i x_{rs}^i\Bigr|}{\binom{n}{2}},
\]
where $(\bm x, \bm \omega)$ is a feasible solution. This measure ranges in $[0,1]$, with values closer to 1 indicating a better approximation of the observed preference matrix. 

Table~\ref{tab:MLOP_sinteticas} reports the computational results for 18 synthetic instances with $n=12$ (from Table~\ref{tab:instance_parameters}). The model was solved for $g \in \{1,2,3,4\}$ with a 4-hour time limit, including the symmetry-breaking constraints \eqref{eq:sim}. The table details the objective value, fit (\%), computational time (s), absolute optimality gap, and best solution group weights. All numerical values are rounded to three decimal places.

\begingroup
\scriptsize
\setlength{\LTleft}{\fill}
\setlength{\LTright}{\fill}

\begin{longtable}[c]{c c l r r r r r l}
\caption{Computational results of the exact method on the synthetic instances of the Mixture Linear Ordering Problem.}
\label{tab:MLOP_sinteticas}\\
\toprule
Instance & $n$ & $\overline{\bm{\omega}}$ & $g$ & Obj.\ value & Fit (\%) & Time (s) & Gap (abs) & Best solution weights \\
\midrule
\endfirsthead

\caption[]{Computational results of the exact method (continued).}\\
\toprule
Instance & $n$ & $\overline{\bm{\omega}}$ & $g$ & Obj.\ value & Fit (\%) & Time (s) & Gap (abs) & Best solution weights \\
\midrule
\endhead

\midrule
\multicolumn{9}{r}{\scriptsize\itshape (continued on next page)}\\
\endfoot

\bottomrule
\endlastfoot

\sffamily\footnotesize

%% --- Bloque R1 ---
\multirow{4}{*}{$R_1$} & \multirow{4}{*}{12} & \multirow{4}{*}{$(0.667,0.333)$} & 1 & 12.144 & 81.600 & 0.005 & 0.000 & (1.000) \\*
 &  &  & 2 & 0.934 & 98.585 & 0.068 & 0.000 & (0.667, 0.333) \\*
 &  &  & 3 & 0.562 & 99.148 & 1.277 & 0.000 & (0.598, 0.333, 0.069) \\*
 &  &  & 4 & 0.290 & 99.561 & 15.177 & 0.000 & (0.520, 0.333, 0.085, 0.062) \\[0.5em]

%% --- Bloque R2 ---
\multirow{4}{*}{$R_2$} & \multirow{4}{*}{12} & \multirow{4}{*}{$(0.500,0.500)$} & 1 & 8.666 & 86.870 & 0.005 & 0.000 & (1.000) \\*
 &  &  & 2 & 1.000 & 98.485 & 0.068 & 0.000 & (0.500, 0.500) \\*
 &  &  & 3 & 0.640 & 99.030 & 2.458 & 0.000 & (0.500, 0.452, 0.048) \\*
 &  &  & 4 & 0.344 & 99.479 & 35.111 & 0.000 & (0.452, 0.452, 0.048, 0.048) \\[0.5em]

%% --- Bloque R3 ---
\multirow{4}{*}{$R_3$} & \multirow{4}{*}{12} & \multirow{4}{*}{$(0.667,0.333)$} & 1 & 12.349 & 81.289 & 0.005 & 0.000 & (1.000) \\*
 &  &  & 2 & 2.874 & 95.645 & 0.311 & 0.000 & (0.667, 0.333) \\*
 &  &  & 3 & 1.942 & 97.058 & 175.647 & 0.000 & (0.523, 0.333, 0.144) \\*
 &  &  & 4 & 1.334 & 97.979 & 14400.000 & 0.603 & (0.386, 0.333, 0.143, 0.138) \\[0.5em]

%% --- Bloque R4 ---
\multirow{4}{*}{$R_4$} & \multirow{4}{*}{12} & \multirow{4}{*}{$(0.500,0.500)$} & 1 & 21.344 & 67.661 & 0.005 & 0.000 & (1.000) \\*
 &  &  & 2 & 2.830 & 95.712 & 0.415 & 0.000 & (0.500, 0.500) \\*
 &  &  & 3 & 1.968 & 97.018 & 421.689 & 0.000 & (0.500, 0.390, 0.110) \\*
 &  &  & 4 & 1.123 & 98.298 & 14400.000 & 0.274 & (0.397, 0.394, 0.106, 0.103) \\[0.5em]

%% --- Bloque R5 ---
\multirow{4}{*}{$R_5$} & \multirow{4}{*}{12} & \multirow{4}{*}{$(0.667,0.333)$} & 1 & 13.604 & 79.388 & 0.006 & 0.000 & (1.000) \\*
 &  &  & 2 & 4.556 & 93.097 & 1.563 & 0.000 & (0.667, 0.333) \\*
 &  &  & 3 & 2.975 & 95.492 & 13930.309 & 0.000 & (0.486, 0.333, 0.181) \\*
 &  &  & 4 & 1.877 & 97.156 & 14400.000 & 1.207 & (0.492, 0.232, 0.179, 0.097) \\[0.5em]

%% --- Bloque R6 ---
\multirow{4}{*}{$R_6$} & \multirow{4}{*}{12} & \multirow{4}{*}{$(0.500,0.500)$} & 1 & 20.440 & 69.030 & 0.005 & 0.000 & (1.000) \\*
 &  &  & 2 & 3.822 & 94.209 & 0.396 & 0.000 & (0.500, 0.500) \\*
 &  &  & 3 & 2.746 & 95.839 & 4996.115 & 0.000 & (0.498, 0.393, 0.109) \\*
 &  &  & 4 & 1.800 & 97.273 & 14400.000 & 1.079 & (0.375, 0.358, 0.145, 0.122) \\[0.5em]

%% --- Bloque R7 ---
\multirow{4}{*}{$R_7$} & \multirow{4}{*}{12} & \multirow{4}{*}{$(0.571,0.286,0.143)$} & 1 & 20.762 & 68.542 & 0.005 & 0.000 & (1.000) \\*
 &  &  & 2 & 5.334 & 91.918 & 2.848 & 0.000 & (0.589, 0.411) \\*
 &  &  & 3 & 0.920 & 98.606 & 6.273 & 0.000 & (0.571, 0.286, 0.143) \\*
 &  &  & 4 & 0.616 & 99.067 & 739.215 & 0.000 & (0.515, 0.286, 0.143, 0.056) \\[0.5em]

%% --- Bloque R8 ---
\multirow{4}{*}{$R_8$} & \multirow{4}{*}{12} & \multirow{4}{*}{$(0.334,0.333,0.333)$} & 1 & 19.305 & 70.750 & 0.005 & 0.000 & (1.000) \\*
 &  &  & 2 & 4.249 & 93.562 & 0.766 & 0.000 & (0.666, 0.334) \\*
 &  &  & 3 & 0.942 & 98.573 & 39.070 & 0.000 & (0.334, 0.333, 0.333) \\*
 &  &  & 4 & 0.676 & 98.976 & 5577.490 & 0.000 & (0.334, 0.333, 0.303, 0.030) \\[0.5em]

%% --- Bloque R9 ---
\multirow{4}{*}{$R_9$} & \multirow{4}{*}{12} & \multirow{4}{*}{$(0.571,0.286,0.143)$} & 1 & 13.853 & 79.011 & 0.005 & 0.000 & (1.000) \\*
 &  &  & 2 & 5.078 & 92.306 & 2.444 & 0.000 & (0.589, 0.411) \\*
 &  &  & 3 & 2.714 & 95.888 & 2907.991 & 0.000 & (0.571, 0.287, 0.142) \\*
 &  &  & 4 & 1.749 & 97.350 & 14400.000 & 0.963 & (0.456, 0.284, 0.143, 0.117) \\[0.5em]

%% --- Bloque R10 ---
\multirow{4}{*}{$R_{10}$} & \multirow{4}{*}{12} & \multirow{4}{*}{$(0.334,0.333,0.333)$} & 1 & 14.416 & 78.158 & 0.006 & 0.000 & (1.000) \\*
 &  &  & 2 & 5.518 & 91.639 & 5.874 & 0.000 & (0.666, 0.334) \\*
 &  &  & 3 & 2.172 & 96.709 & 1671.735 & 0.000 & (0.342, 0.336, 0.322) \\*
 &  &  & 4 & 1.480 & 97.758 & 14400.000 & 0.734 & (0.335, 0.332, 0.249, 0.084) \\[0.5em]

%% --- Bloque R11 ---
\multirow{4}{*}{$R_{11}$} & \multirow{4}{*}{12} & \multirow{4}{*}{$(0.571,0.286,0.143)$} & 1 & 15.299 & 76.820 & 0.005 & 0.000 & (1.000) \\*
 &  &  & 2 & 5.179 & 92.153 & 3.247 & 0.000 & (0.602, 0.398) \\*
 &  &  & 3 & 3.494 & 94.706 & 14400.000 & 0.994 & (0.444, 0.395, 0.161) \\*
 &  &  & 4 & 1.971 & 97.014 & 14400.000 & 1.313 & (0.432, 0.276, 0.168, 0.124) \\[0.5em]

%% --- Bloque R12 ---
\multirow{4}{*}{$R_{12}$} & \multirow{4}{*}{12} & \multirow{4}{*}{$(0.334,0.333,0.333)$} & 1 & 22.301 & 66.211 & 0.005 & 0.000 & (1.000) \\*
 &  &  & 2 & 4.887 & 92.595 & 3.924 & 0.000 & (0.653, 0.347) \\*
 &  &  & 3 & 3.267 & 95.050 & 14400.000 & 1.480 & (0.376, 0.334, 0.290) \\*
 &  &  & 4 & 2.432 & 96.315 & 14400.000 & 2.083 & (0.336, 0.316, 0.253, 0.095) \\[0.5em]

%% --- Bloque R13 ---
\multirow{4}{*}{$R_{13}$} & \multirow{4}{*}{12} & \multirow{4}{*}{$(0.533,0.257,0.133,0.067)$} & 1 & 18.700 & 71.667 & 0.005 & 0.000 & (1.000) \\*
 &  &  & 2 & 6.624 & 89.964 & 2.936 & 0.000 & (0.720, 0.280) \\*
 &  &  & 3 & 1.811 & 97.256 & 186.039 & 0.000 & (0.533, 0.267, 0.200) \\*
 &  &  & 4 & 0.676 & 98.976 & 346.974 & 0.000 & (0.533, 0.267, 0.133, 0.067) \\[0.5em]

%% --- Bloque R14 ---
\multirow{4}{*}{$R_{14}$} & \multirow{4}{*}{12} & \multirow{4}{*}{$(0.250,0.250,0.250,0.250)$} & 1 & 19.850 & 69.924 & 0.005 & 0.000 & (1.000) \\*
 &  &  & 2 & 7.182 & 89.118 & 20.508 & 0.000 & (0.511, 0.489) \\*
 &  &  & 3 & 2.296 & 96.521 & 2145.983 & 0.000 & (0.500, 0.251, 0.249) \\*
 &  &  & 4 & 0.794 & 98.797 & 9312.460 & 0.000 & (0.250, 0.250, 0.250, 0.250) \\[0.5em]

%% --- Bloque R15 ---
\multirow{4}{*}{$R_{15}$} & \multirow{4}{*}{12} & \multirow{4}{*}{$(0.533,0.257,0.133,0.067)$} & 1 & 17.957 & 72.792 & 0.005 & 0.000 & (1.000) \\*
 &  &  & 2 & 6.751 & 89.771 & 6.888 & 0.000 & (0.588, 0.412) \\*
 &  &  & 3 & 3.009 & 95.441 & 14400.000 & 0.681 & (0.536, 0.275, 0.189) \\*
 &  &  & 4 & 1.894 & 97.130 & 14400.000 & 1.044 & (0.405, 0.269, 0.195, 0.131) \\[0.5em]

%% --- Bloque R16 ---
\multirow{4}{*}{$R_{16}$} & \multirow{4}{*}{12} & \multirow{4}{*}{$(0.250,0.250,0.250,0.250)$} & 1 & 19.944 & 69.782 & 0.005 & 0.000 & (1.000) \\*
 &  &  & 2 & 5.816 & 91.188 & 3.168 & 0.000 & (0.692, 0.308) \\*
 &  &  & 3 & 2.864 & 95.661 & 14400.000 & 0.705 & (0.498, 0.254, 0.248) \\*
 &  &  & 4 & 1.862 & 97.179 & 14400.000 & 1.408 & (0.384, 0.252, 0.250, 0.114) \\[0.5em]

%% --- Bloque R17 ---
\multirow{4}{*}{$R_{17}$} & \multirow{4}{*}{12} & \multirow{4}{*}{$(0.533,0.257,0.133,0.067)$} & 1 & 18.728 & 71.624 & 0.005 & 0.000 & (1.000) \\*
 &  &  & 2 & 6.833 & 89.647 & 14.876 & 0.000 & (0.663, 0.337) \\*
 &  &  & 3 & 3.212 & 95.133 & 14400.000 & 1.284 & (0.596, 0.234, 0.170) \\*
 &  &  & 4 & 2.289 & 96.532 & 14400.000 & 1.763 & (0.411, 0.311, 0.186, 0.092) \\[0.5em]

%% --- Bloque R18 ---
\multirow{4}{*}{$R_{18}$} & \multirow{4}{*}{12} & \multirow{4}{*}{$(0.250,0.250,0.250,0.250)$} & 1 & 21.512 & 67.406 & 0.005 & 0.000 & (1.000) \\*
 &  &  & 2 & 6.072 & 90.800 & 14.748 & 0.000 & (0.641, 0.359) \\*
 &  &  & 3 & 3.070 & 95.348 & 14400.000 & 1.374 & (0.506, 0.278, 0.216) \\*
 &  &  & 4 & 2.178 & 96.700 & 14400.000 & 1.876 & (0.324, 0.252, 0.250, 0.174) \\

\end{longtable}
\endgroup

As observed in Table~\ref{tab:MLOP_sinteticas}, increasing the number of latent groups $g$ significantly improves the representation of the aggregate preference matrix. For $g=1$, which corresponds to the classical Linear Ordering Problem (assuming a strictly homogeneous population), the Fit ranges from a minimum of 66.211\% ($R_{12}$) to a maximum of 86.870\% ($R_2$). By relaxing the homogeneity assumption and allowing the model to capture multiple preference patterns ($g>1$), the Fit metric increases substantially. With $g=2$, all instances exceed an 89\% fit, and for $g=4$, the model systematically achieves values above 96\%, indicating that the observed matrix is almost perfectly explained by the mixture of underlying linear orders. Naturally, this is accompanied by a monotonic decrease in the objective value.

More importantly, the magnitude of this reduction in the objective value is highly indicative of the true underlying population structure. We observe a sharp drop in the objective function, similar to the well-known \emph{elbow effect} in clustering, exactly when $g$ matches the true number of latent groups used to generate the instance. This drop is then followed by only marginal improvements for larger values of $g$. For example, instance $R_1$ was generated with 2 latent groups. Its objective value decreases significantly from 12.144 for $g=1$ to 0.934 for $g=2$, but it only improves by a residual 0.372 and 0.272 when allowing $g=3$ and $g=4$, respectively. Conversely, for instances generated with 4 latent groups, such as $R_{13}$, the substantial decreases continue all the way up to $g=4$. It should be noted, however, that the sharpness of this elbow effect can be mitigated by the level of noise in the data; in highly noisy instances, the transition may appear less abrupt as the model begins to leverage additional groups to fit random variations rather than true structural patterns. Despite this, the overall behavior confirms that the proposed exact formulation not only minimizes the preference disagreement but also serves as a reliable mechanism to deduce the unknown number of consensus rankings in a given population.

However, the inclusion of multiple preference groups imposes a severe penalty on computational tractability. For $g=1$, the solver trivially proves optimality in a few milliseconds. Moving to $g=2$, the computational time increases but remains entirely manageable, with all instances solved to proven optimality in less than 21 seconds. The combinatorial complexity grows drastically for $g=3$ and $g=4$. While some structurally simpler instances are solved within seconds or minutes for $g=3$, a significant portion hits the 4-hour time limit (14,400 s). For $g=4$, only a few instances ($R_1, R_2, R_7, R_8, R_{13}, R_{14}$) are solved to optimality before the timer expires. This exponential growth in computational effort highlights a severe scalability limitation of the exact MIP formulation. Given that real-world applications frequently require handling larger matrices and multiple latent groups, relying solely on the exact approach becomes impractical. This computational bottleneck strongly motivates the development of the alternating-direction heuristic proposed in Section~\ref{sec:heuristic}, which aims to bypass these limitations.

Despite failing to prove optimality within 4 hours for the more complex cases, the exact formulation proves to be highly effective at finding near-optimal solutions. The absolute optimality gaps remaining after the time limit are remarkably small. For $g=3$, the maximum absolute gap observed is 1.480 (instance $R_{12}$), and for $g=4$, it peaks at 2.083 (also for $R_{12}$). These marginal gaps confirm that, even when exact convergence is computationally prohibitive, the proposed formulation yields exceptionally high-quality approximations of the global optimum, thereby establishing a solid mathematical benchmark against which the performance of our heuristic can be confidently evaluated.

Finally, an analysis of the estimated group weights provides valuable insight into the mixture structure and the risk of overfitting. When the allowed number of groups $g$ perfectly matches the true number of latent groups used to generate the synthetic data (indicated by the size of the $\omega$ vector), the exact model accurately recovers the true population proportions (e.g., yielding weights of exactly 0.500 and 0.500 for $R_2$ and $R_4$ with $g=2$). However, when $g$ exceeds the true number of underlying groups, the model leverages the extra degrees of freedom to fit the random noise present in the observed matrix. This overfitting manifests as a fragmentation of the true groups. For instance, in $R_4$ ($\omega = (0.500, 0.500)$), setting $g=4$ artificially splits the population into four weights (0.397, 0.394, 0.106, 0.103) to further lower the objective value from 2.830 to 1.123. This highlights a classic bias-variance trade-off: while increasing $g$ monotonically improves the Fit metric, artificially high values of $g$ result in memorizing the noise rather than capturing the true underlying preference structure.

\subsection{Comparison between the exact method and the heuristic}

This section compares the exact mixed-integer formulation with the proposed matheuristic approach. The primary goal is to evaluate the heuristic's ability to maintain high solution quality while mitigating the severe computational burden associated with the exact model.

Table~\ref{tab:exact_vs_heuristic} reports the results for all 36 synthetic instances solved for $g \in \{2,3,4\}$. Both approaches enforce the symmetry-breaking constraints \eqref{eq:sim}. For both methods, we report the objective value, fit (\%), and computational time (s). For the exact method, which was run with a global 4-hour (14,400 seconds) time limit, we additionally report the absolute optimality gap at termination. All numerical values are rounded to three decimal places.

To ensure a fair comparison, the heuristic's hyper-parameters were explicitly designed to bound its maximum theoretical effort to the exact same 4-hour limit. To reduce dependence on the initial weight configuration, the heuristic was run from 10 completely random starting points. For each starting point, the maximum number of alternating iterations was capped at 12, and the internal Mixed-Integer Linear Programming (MILP) subproblem was constrained by a 120-second time limit. The algorithm advances if the internal convergence tolerance between consecutive solutions is greater than $10^{-5}$. Consequently, in the absolute worst-case scenario, the heuristic would spend exactly $10 \cdot 12 \cdot 120 = 14400$ seconds resolving MILP subproblems, perfectly matching the exact method's time limit. For the heuristic, we also report the total number of alternating iterations executed across all 10 initializations.

\begingroup
\footnotesize
\setlength{\LTleft}{\fill}
\setlength{\LTright}{\fill}

\begin{longtable}{c c r r r r r r r r}
\caption{Comparison between the exact method and the heuristic on the synthetic instances.}
\label{tab:exact_vs_heuristic}\\
\toprule
Inst. & $g$ & \multicolumn{4}{c}{Exact} & \multicolumn{4}{c}{Heuristic} \\
\cmidrule(lr){3-6} \cmidrule(lr){7-10}
 & & Obj. Value & Fit (\%) & Time (s) & Gap (abs) & Obj. Value & Fit (\%) & Time (s) & \# Iter. \\
\midrule
\endfirsthead

\caption[]{Comparison between the exact method and the heuristic (continued).}\\
\toprule
Inst. & $g$ & \multicolumn{4}{c}{Exact} & \multicolumn{4}{c}{Heuristic} \\
\cmidrule(lr){3-6} \cmidrule(lr){7-10}
 & & Obj. Value & Fit (\%) & Time (s) & Gap (abs) & Obj. Value & Fit (\%) & Time (s) & \# Iter. \\
\midrule
\endhead

\midrule
\multicolumn{10}{r}{\scriptsize\itshape (continues on next page)}\\
\endfoot

\bottomrule
\endlastfoot

\sffamily\footnotesize

% ==========================================
% GRUPO 1: n=12 (C=66)
% ==========================================
\multirow{3}{*}{$R_1$} & 2 & 0.934 & 98.585 & 0.068 & 0.000 & 0.934 & 98.585 & 0.350 & 25 \\*
 & 3 & 0.562 & 99.148 & 1.277 & 0.000 & 0.562 & 99.148 & 0.711 & 30 \\*
 & 4 & 0.290 & 99.561 & 15.177 & 0.000 & 0.290 & 99.561 & 1.770 & 31 \\[0.5em]

\multirow{3}{*}{$R_2$} & 2 & 1.000 & 98.485 & 0.068 & 0.000 & 1.000 & 98.485 & 0.334 & 24 \\*
 & 3 & 0.640 & 99.030 & 2.458 & 0.000 & 0.640 & 99.030 & 1.032 & 28 \\*
 & 4 & 0.344 & 99.479 & 35.111 & 0.000 & 0.344 & 99.479 & 2.481 & 39 \\[0.5em]

\multirow{3}{*}{$R_3$} & 2 & 2.874 & 95.645 & 0.311 & 0.000 & 2.874 & 95.645 & 0.371 & 26 \\*
 & 3 & 1.942 & 97.058 & 175.647 & 0.000 & 1.942 & 97.058 & 1.427 & 32 \\*
 & 4 & \textbf{1.334} & 97.979 & 14400.000 & 0.603 & 1.340 & 97.970 & 7.945 & 43 \\[0.5em]

\multirow{3}{*}{$R_4$} & 2 & 2.830 & 95.712 & 0.415 & 0.000 & 2.830 & 95.712 & 0.440 & 26 \\*
 & 3 & 1.968 & 97.018 & 421.689 & 0.000 & 1.968 & 97.018 & 4.132 & 33 \\*
 & 4 & \textbf{1.123} & 98.298 & 14400.000 & 0.274 & 1.131 & 98.286 & 11.302 & 35 \\[0.5em]

\multirow{3}{*}{$R_5$} & 2 & 4.556 & 93.097 & 1.563 & 0.000 & 4.556 & 93.097 & 0.587 & 32 \\*
 & 3 & 2.975 & 95.492 & 13930.309 & 0.000 & 2.975 & 95.492 & 2.281 & 36 \\*
 & 4 & 1.877 & 97.156 & 14400.000 & 1.207 & 1.877 & 97.156 & 15.678 & 41 \\[0.5em]

\multirow{3}{*}{$R_6$} & 2 & 3.822 & 94.209 & 0.396 & 0.000 & 3.822 & 94.209 & 0.611 & 32 \\*
 & 3 & 2.746 & 95.839 & 4996.115 & 0.000 & 2.746 & 95.839 & 6.691 & 44 \\*
 & 4 & 1.800 & 97.273 & 14400.000 & 1.079 & \textbf{1.781} & 97.302 & 21.476 & 42 \\[0.5em]

\multirow{3}{*}{$R_7$} & 2 & 5.334 & 91.918 & 2.848 & 0.000 & 5.334 & 91.918 & 0.465 & 32 \\*
 & 3 & 0.920 & 98.606 & 6.273 & 0.000 & 0.920 & 98.606 & 1.383 & 33 \\*
 & 4 & 0.616 & 99.067 & 739.215 & 0.000 & 0.616 & 99.067 & 4.445 & 28 \\[0.5em]

\multirow{3}{*}{$R_8$} & 2 & 4.249 & 93.562 & 0.766 & 0.000 & 4.249 & 93.562 & 0.447 & 27 \\*
 & 3 & 0.942 & 98.573 & 39.070 & 0.000 & 0.942 & 98.573 & 3.013 & 37 \\*
 & 4 & 0.676 & 98.976 & 5577.490 & 0.000 & 0.676 & 98.976 & 12.481 & 35 \\[0.5em]

\multirow{3}{*}{$R_9$} & 2 & 5.078 & 92.306 & 2.444 & 0.000 & 5.078 & 92.306 & 0.618 & 37 \\*
 & 3 & 2.714 & 95.888 & 2907.991 & 0.000 & 2.714 & 95.888 & 2.245 & 44 \\*
 & 4 & \textbf{1.749} & 97.350 & 14400.000 & 0.963 & 1.754 & 97.342 & 11.593 & 38 \\[0.5em]

\multirow{3}{*}{$R_{10}$} & 2 & 5.518 & 91.639 & 5.874 & 0.000 & 5.518 & 91.639 & 0.572 & 30 \\*
 & 3 & 2.172 & 96.709 & 1671.735 & 0.000 & 2.172 & 96.709 & 2.688 & 33 \\*
 & 4 & 1.480 & 97.758 & 14400.000 & 0.734 & 1.480 & 97.758 & 17.786 & 43 \\[0.5em]

\multirow{3}{*}{$R_{11}$} & 2 & 5.179 & 92.153 & 3.247 & 0.000 & 5.179 & 92.153 & 0.672 & 38 \\*
 & 3 & 3.494 & 94.706 & 14400.000 & 0.994 & 3.494 & 94.706 & 2.547 & 35 \\*
 & 4 & 1.971 & 97.014 & 14400.000 & 1.313 & 1.971 & 97.014 & 19.360 & 48 \\[0.5em]

\multirow{3}{*}{$R_{12}$} & 2 & 4.887 & 92.595 & 3.924 & 0.000 & 4.887 & 92.595 & 0.353 & 22 \\*
 & 3 & 3.267 & 95.050 & 14400.000 & 1.480 & 3.267 & 95.050 & 3.189 & 33 \\*
 & 4 & 2.432 & 96.315 & 14400.000 & 2.083 & \textbf{2.365} & 96.417 & 127.782 & 42 \\[0.5em]

\multirow{3}{*}{$R_{13}$} & 2 & 6.624 & 89.964 & 2.936 & 0.000 & 6.624 & 89.964 & 0.419 & 29 \\*
 & 3 & 1.811 & 97.256 & 186.039 & 0.000 & 1.811 & 97.256 & 1.591 & 33 \\*
 & 4 & 0.676 & 98.976 & 346.974 & 0.000 & 0.676 & 98.976 & 5.762 & 31 \\[0.5em]

\multirow{3}{*}{$R_{14}$} & 2 & 7.182 & 89.118 & 20.508 & 0.000 & 7.182 & 89.118 & 0.525 & 27 \\*
 & 3 & 2.296 & 96.521 & 2145.983 & 0.000 & 2.296 & 96.521 & 5.386 & 29 \\*
 & 4 & 0.794 & 98.797 & 9312.460 & 0.000 & 0.794 & 98.797 & 30.096 & 38 \\[0.5em]

\multirow{3}{*}{$R_{15}$} & 2 & 6.751 & 89.771 & 6.888 & 0.000 & 6.751 & 89.771 & 0.792 & 45 \\*
 & 3 & 3.009 & 95.441 & 14400.000 & 0.681 & 3.009 & 95.441 & 2.585 & 36 \\*
 & 4 & \textbf{1.894} & 97.130 & 14400.000 & 1.044 & 1.895 & 97.129 & 13.525 & 38 \\[0.5em]

\multirow{3}{*}{$R_{16}$} & 2 & 5.816 & 91.188 & 3.168 & 0.000 & 5.816 & 91.188 & 0.530 & 30 \\*
 & 3 & 2.864 & 95.661 & 14400.000 & 0.705 & 2.864 & 95.661 & 4.032 & 36 \\*
 & 4 & 1.862 & 97.179 & 14400.000 & 1.408 & 1.862 & 97.179 & 37.439 & 45 \\[0.5em]

\multirow{3}{*}{$R_{17}$} & 2 & 6.833 & 89.647 & 14.876 & 0.000 & 6.833 & 89.647 & 0.769 & 40 \\*
 & 3 & 3.212 & 95.133 & 14400.000 & 1.284 & 3.212 & 95.133 & 2.511 & 36 \\*
 & 4 & 2.289 & 96.532 & 14400.000 & 1.763 & \textbf{2.157} & 96.732 & 19.173 & 32 \\[0.5em]

\multirow{3}{*}{$R_{18}$} & 2 & 6.072 & 90.800 & 14.748 & 0.000 & 6.072 & 90.800 & 0.464 & 27 \\*
 & 3 & 3.070 & 95.348 & 14400.000 & 1.374 & 3.070 & 95.348 & 7.308 & 46 \\*
 & 4 & 2.178 & 96.700 & 14400.000 & 1.876 & \textbf{2.068} & 96.867 & 72.218 & 53 \\[0.5em]

% ==========================================
% GRUPO 2: n=24 (C=276)
% ==========================================
\multirow{3}{*}{$R_{19}$} & 2 & 2.608 & 99.055 & 1.808 & 0.000 & 2.608 & 99.055 & 2.026 & 22 \\*
 & 3 & 1.830 & 99.337 & 14400.000 & 0.714 & 1.830 & 99.337 & 9.421 & 35 \\*
 & 4 & 1.441 & 99.478 & 14400.000 & 1.188 & \textbf{1.178} & 99.573 & 72.684 & 47 \\[0.5em]

\multirow{3}{*}{$R_{20}$} & 2 & 2.472 & 99.104 & 2.086 & 0.000 & 2.472 & 99.104 & 2.676 & 26 \\*
 & 3 & 1.679 & 99.392 & 8083.952 & 0.000 & 1.679 & 99.392 & 16.551 & 30 \\*
 & 4 & 0.918 & 99.667 & 14400.000 & 0.718 & 0.918 & 99.667 & 97.580 & 38 \\[0.5em]

\multirow{3}{*}{$R_{21}$} & 2 & 9.366 & 96.607 & 43.662 & 0.000 & 9.366 & 96.607 & 3.406 & 30 \\*
 & 3 & 6.649 & 97.591 & 14400.000 & 3.656 & 6.649 & 97.591 & 19.513 & 34 \\*
 & 4 & 5.086 & 98.157 & 14400.000 & 4.864 & \textbf{4.890} & 98.228 & 276.579 & 42 \\[0.5em]

\multirow{3}{*}{$R_{22}$} & 2 & 10.012 & 96.372 & 70.229 & 0.000 & 10.012 & 96.372 & 3.923 & 30 \\*
 & 3 & 7.642 & 97.231 & 14400.000 & 4.613 & \textbf{7.402} & 97.318 & 145.041 & 38 \\*
 & 4 & 4.953 & 98.205 & 14400.000 & 4.792 & 4.953 & 98.205 & 2714.293 & 50 \\[0.5em]

\multirow{3}{*}{$R_{23}$} & 2 & 13.570 & 95.083 & 528.748 & 0.000 & 13.570 & 95.083 & 3.583 & 30 \\*
 & 3 & 9.149 & 96.685 & 14400.000 & 6.350 & 9.149 & 96.685 & 43.849 & 42 \\*
 & 4 & 6.925 & 97.491 & 14400.000 & 6.775 & 6.925 & 97.491 & 1378.509 & 71 \\[0.5em]

\multirow{3}{*}{$R_{24}$} & 2 & 13.928 & 94.954 & 136.498 & 0.000 & 13.928 & 94.954 & 6.171 & 31 \\*
 & 3 & 10.358 & 96.247 & 14400.000 & 7.886 & \textbf{10.160} & 96.319 & 246.704 & 43 \\*
 & 4 & 7.106 & 97.425 & 14400.000 & 6.972 & 7.106 & 97.425 & 4423.650 & 55 \\[0.5em]

\multirow{3}{*}{$R_{25}$} & 2 & 19.301 & 93.007 & 947.809 & 0.000 & 19.301 & 93.007 & 3.752 & 38 \\*
 & 3 & 2.657 & 99.037 & 14400.000 & 0.155 & 2.657 & 99.037 & 25.132 & 38 \\*
 & 4 & 1.978 & 99.283 & 14400.000 & 1.760 & 1.978 & 99.283 & 164.933 & 35 \\[0.5em]

\multirow{3}{*}{$R_{26}$} & 2 & 22.114 & 91.988 & 1151.716 & 0.000 & 22.114 & 91.988 & 7.922 & 30 \\*
 & 3 & 2.655 & 99.038 & 14400.000 & 0.134 & 2.655 & 99.038 & 134.243 & 33 \\*
 & 4 & 2.127 & 99.229 & 14400.000 & 1.857 & \textbf{2.079} & 99.247 & 959.355 & 34 \\[0.5em]

\multirow{3}{*}{$R_{27}$} & 2 & 21.854 & 92.082 & 672.416 & 0.000 & 21.854 & 92.082 & 6.140 & 46 \\*
 & 3 & 8.293 & 96.995 & 14400.000 & 6.338 & 8.293 & 96.995 & 36.696 & 42 \\*
 & 4 & 5.795 & 97.900 & 14400.000 & 5.632 & 5.795 & 97.900 & 597.589 & 54 \\[0.5em]

\multirow{3}{*}{$R_{28}$} & 2 & 24.664 & 91.064 & 1125.835 & 0.000 & 24.664 & 91.064 & 10.402 & 31 \\*
 & 3 & 8.767 & 96.824 & 14400.000 & 6.510 & 8.767 & 96.824 & 434.183 & 34 \\*
 & 4 & 7.121 & 97.420 & 14400.000 & 7.115 & \textbf{6.841} & 97.521 & 4380.698 & 53 \\[0.5em]

\multirow{3}{*}{$R_{29}$} & 2 & 25.193 & 90.872 & 834.477 & 0.000 & 25.193 & 90.872 & 7.659 & 48 \\*
 & 3 & 11.226 & 95.933 & 14400.000 & 9.592 & 11.226 & 95.933 & 92.822 & 49 \\*
 & 4 & 9.430 & 96.583 & 14400.000 & 9.430 & \textbf{7.916} & 97.132 & 4620.861 & 53 \\[0.5em]

\multirow{3}{*}{$R_{30}$} & 2 & 29.031 & 89.482 & 14400.000 & 1.518 & 29.031 & 89.482 & 12.492 & 33 \\*
 & 3 & 12.641 & 95.420 & 14400.000 & 10.660 & \textbf{12.629} & 95.424 & 2115.681 & 59 \\*
 & 4 & 9.651 & 96.503 & 14400.000 & 9.651 & 9.651 & 96.503 & 7327.513 & 61 \\[0.5em]

\multirow{3}{*}{$R_{31}$} & 2 & 24.223 & 91.224 & 815.940 & 0.000 & 24.223 & 91.224 & 4.682 & 32 \\*
 & 3 & 9.759 & 96.464 & 14400.000 & 8.043 & 9.759 & 96.464 & 73.305 & 47 \\*
 & 4 & 2.408 & 99.128 & 14400.000 & 2.327 & 2.408 & 99.128 & 895.978 & 50 \\[0.5em]

\multirow{3}{*}{$R_{32}$} & 2 & 32.916 & 88.074 & 9425.741 & 0.000 & 32.916 & 88.074 & 12.226 & 38 \\*
 & 3 & 13.440 & 95.130 & 14400.000 & 10.924 & \textbf{12.768} & 95.374 & 3510.304 & 39 \\*
 & 4 & 2.510 & 99.091 & 14400.000 & 2.498 & 2.510 & 99.091 & 4798.261 & 45 \\[0.5em]

\multirow{3}{*}{$R_{33}$} & 2 & 22.594 & 91.814 & 612.998 & 0.000 & 22.594 & 91.814 & 6.014 & 41 \\*
 & 3 & 11.413 & 95.865 & 14400.000 & 8.898 & 11.413 & 95.865 & 85.769 & 52 \\*
 & 4 & 9.909 & 96.410 & 14400.000 & 9.834 & \textbf{7.533} & 97.271 & 2000.465 & 45 \\[0.5em]

\multirow{3}{*}{$R_{34}$} & 2 & 33.696 & 87.791 & 14400.000 & 1.830 & 33.696 & 87.791 & 9.183 & 28 \\*
 & 3 & 15.230 & 94.482 & 14400.000 & 12.836 & \textbf{15.016} & 94.559 & 5935.305 & 60 \\*
 & 4 & 8.892 & 96.778 & 14400.000 & 8.891 & \textbf{7.779} & 97.182 & 7181.446 & 67 \\[0.5em]

\multirow{3}{*}{$R_{35}$} & 2 & 25.837 & 90.639 & 759.693 & 0.000 & 25.837 & 90.639 & 5.531 & 35 \\*
 & 3 & 17.244 & 93.752 & 14400.000 & 15.785 & \textbf{15.171} & 94.503 & 153.991 & 42 \\*
 & 4 & 13.261 & 95.195 & 14400.000 & 13.260 & \textbf{10.012} & 96.372 & 4766.097 & 46 \\[0.5em]

\multirow{3}{*}{$R_{36}$} & 2 & 29.543 & 89.296 & 7145.818 & 0.000 & 29.543 & 89.296 & 7.655 & 32 \\*
 & 3 & 17.952 & 93.496 & 14400.000 & 16.544 & \textbf{16.204} & 94.129 & 5067.869 & 56 \\*
 & 4 & 12.677 & 95.407 & 14400.000 & 12.677 & \textbf{11.917} & 95.682 & 6486.656 & 54 \\[0.5em]

\end{longtable}
\endgroup

As shown in Table~\ref{tab:exact_vs_heuristic}, the proposed alternating-direction matheuristic vastly outperforms the compact exact formulation in terms of computational efficiency and robustness, especially as the instance size ($n=24$) and the number of groups ($g$) increase. The table clearly illustrates the competitive performance of the heuristic on small instances. Specifically, for $g \in \{2, 3\}$, the heuristic perfectly matches the exact solution in all cases while requiring significantly less computational time. For $g=4$, the heuristic actually surpasses the exact solver in several instances ($R_{6}$, $R_{12}$, $R_{17}$, and $R_{18}$) where the exact method hits the time limit before closing the optimality gap. Conversely, while the exact method achieves marginally better fits in a few other cases ($R_{3}$, $R_{4}$, $R_{9}$, and $R_{15}$ with $g=4$), it is worth noting that it still hits the time limit without proving optimality in some of these instances as well. More importantly, for larger instances, while the exact solver systematically hits the time limit for $g \in \{3,4\}$ and fails to prove optimality (leaving significant gaps), the heuristic consistently delivers higher-quality solutions in a fraction of the computing time. This robust performance demonstrates the heuristic's capability to effectively navigate the complex combinatorial space of the mixture model without the prohibitive computational cost.

Furthermore, the matheuristic allows us to confirm that the elbow effect observed in the exact method for small instances holds true for larger and more complex populations. Because the exact method struggles to converge for $g \in \{3,4\}$ in these larger cases, the true optimal objective values remain unknown. The proposed matheuristic, however, successfully reaches tight upper bounds, revealing the same sharp drop in the objective function when $g$ matches the true number of latent groups. For example, in instance $R_{25}$ (generated with 3 groups), the objective value drops significantly until reaching 2.657 at $g=3$, followed by a much less pronounced improvement at $g=4$ (1.978). As noted previously, while high levels of noise can slightly smooth this transition, the heuristic's robust performance confirms that this approach acts as a highly reliable mechanism to deduce the unknown number of preference profiles, even in large-scale scenarios where exact solvers fail.

Consequently, this scalable approach bridges the gap between theoretical modeling and practical application. By overcoming the computational bottlenecks of the exact formulation without sacrificing solution quality or structural interpretability, the matheuristic provides a powerful and necessary tool for analyzing real-world preference data, where large numbers of items and diverse user profiles are the norm.

\section{Application to sushi preferences} \label{sec:sushi}

In this section, we apply our proposed methodology to the SUSHI Preference Data Set \citep{kamishima2003nantonac}, a benchmark dataset containing 5000 individual rankings over 10 sushi types, originally introduced to evaluate preference clustering algorithms.\footnote{The aggregated sushi instance, model files, and corresponding computational results are also available in the supplementary repository.} Since our approach operates exclusively on aggregated data, we first compile the observed individual rankings into the pairwise preference count matrix $A = (a_{rs})_{n \times n}$, where $a_{rs}$ denotes the number of respondents who rank sushi type $r$ above sushi type $s$. From this matrix, we construct the normalized cost matrix $C = (c_{rs})_{n \times n}$, whose entries are defined as
\begin{equation} 
    c_{rs} = \frac{a_{rs}}{a_{rs} + a_{sr}}, 
\end{equation} 
so that $c_{rs}$ represents the proportion of respondents who prefer item $r$ to item $s$.

To solve the underlying model, we apply our alternating-direction matheuristic to matrix $C$, running 20 multi-start initializations for each cluster size $g \in \{1, \dots, 10\}$. Table~\ref{tab:sushi_elbow} and Figure~\ref{fig:sushi_elbow_plot} summarize the evolution of the objective value, the cumulative reduction relative to the homogeneous base case ($g=1$), and the overall model fit. To visualize the structural improvements effectively, Figure~\ref{fig:sushi_elbow_plot} displays both the full trajectory of the objective value (left panel) and a truncated view starting from $g=2$ (right panel). This truncation removes the extreme scale distortion caused by the rigid single-consensus assumption, allowing for a finer analysis of the multi-group dynamics.

Choosing the appropriate number of groups ($g$) requires balancing model accuracy with simplicity. While increasing $g$ naturally allows the model to fit the aggregated data better, using too many groups makes it overly complex. This risks capturing random statistical noise rather than meaningful preference patterns, a problem known as overfitting. To avoid this, we analyze the marginal reduction in the objective value as we add new groups, applying a standard elbow criterion to its trajectory. As Figure~\ref{fig:sushi_elbow_plot} shows, moving away from a single global consensus ($g=1$) drastically reduces the objective value. Looking at the adjusted trajectory (right panel), the reductions remain substantial up to $g=4$. Specifically, the transition from $g=3$ to $g=4$ still yields a 44.033\% relative drop in the objective value (reaching a cumulative reduction of over 92\% and an absolute fit of 97.310\%). After $g=4$, the objective value curve flattens noticeably; adding more groups beyond this point provides very small reductions and directly increases the risk of overfitting. Therefore, we select $g=4$ as a practical configuration that successfully balances accuracy with interpretability for our analysis.

\begin{table}[htp!]
\scriptsize
\centering
\caption{Evolution of the objective value, drops, and absolute fit for the SUSHI dataset as $g$ increases.}
\label{tab:sushi_elbow}
\begin{tabular}{crrrr}
\toprule
$g$ & Objective Value & Relative Drop (\%) & Cumulative Drop (\%) & Fit (\%) \\
\midrule
1  & 15.390 & --     & --  & 65.800 \\
2  & 4.253  & 72.365 & 72.365 & 90.550 \\
3  & 2.162  & 49.165 & 85.952 & 95.200 \\
4  & 1.210  & 44.033 & 92.138 & 97.310 \\
5  & 0.734  & 39.339 & 95.231 & 98.370 \\
6  & 0.511  & 30.381 & 96.680 & 98.690 \\
7  & 0.341  & 33.268 & 97.784 & 99.230 \\
8  & 0.253  & 25.806 & 98.356 & 99.420 \\
9  & 0.175  & 30.830 & 98.863 & 99.610 \\
10 & 0.129  & 26.286 & 99.162 & 99.660 \\
\bottomrule
\end{tabular}
\end{table}

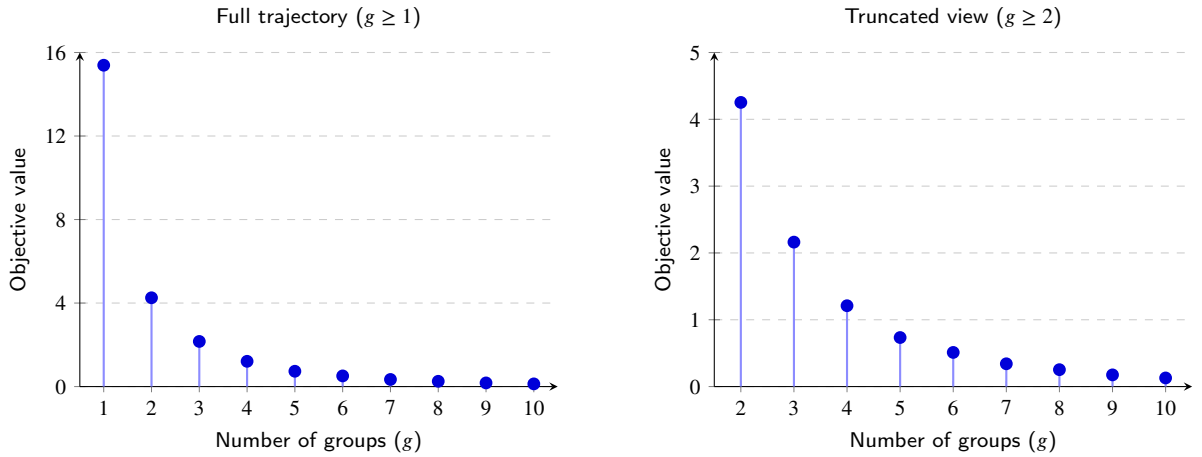
\begin{figure}
    \centering
    \scriptsize
    \caption{Evolution of the objective value for the SUSHI dataset. The left panel displays the full trajectory starting from $g=1$, while the right panel provides a truncated view starting from $g=2$ to better illustrate the structural improvements without the scale distortion of the homogeneous case.}
    \begin{minipage}{0.48\textwidth}
        \centering
        \begin{tikzpicture}
        \begin{axis}[
            width=\linewidth, height=6cm,
            title={Full trajectory ($g \geq 1$)},
            xlabel={Number of groups ($g$)}, 
            ylabel={Objective value},
            xmin=0.5, xmax=10.5, ymin=0, ymax=16,
            xtick={1,2,3,4,5,6,7,8,9,10}, ytick={0,4,8,12,16},
            axis x line=bottom, axis y line=left,
            ymajorgrids=true, grid style={dashed, gray!40}
        ]
        \addplot+[
            ycomb,
            thick,
            draw=blue!45,
            mark=*,
            mark options={
                fill=blue!85!black,
                draw=blue!85!black
            }
        ] coordinates {
            (1,15.390)(2,4.253)(3,2.162)(4,1.210)(5,0.734)(6,0.511)(7,0.341)(8,0.253)(9,0.175)(10,0.129)
        };
        \end{axis}
        \end{tikzpicture}
    \end{minipage}\hfill
    \begin{minipage}{0.48\textwidth}
        \centering
        \begin{tikzpicture}
        \begin{axis}[
            width=\linewidth, height=6cm,
            title={Truncated view ($g \geq 2$)},
            xlabel={Number of groups ($g$)}, 
            ylabel={Objective value},
            xmin=1.5, xmax=10.5, ymin=0, ymax=5,
            xtick={2,3,4,5,6,7,8,9,10}, ytick={0,1,2,3,4,5},
            axis x line=bottom, axis y line=left,
            ymajorgrids=true, grid style={dashed, gray!40}
        ]
        \addplot+[
            ycomb,
            thick,
            draw=blue!45,
            mark=*,
            mark options={
                fill=blue!85!black,
                draw=blue!85!black
            }
        ] coordinates {
            (2,4.253)(3,2.162)(4,1.210)(5,0.734)(6,0.511)(7,0.341)(8,0.253)(9,0.175)(10,0.129)
        };
        \end{axis}
        \end{tikzpicture}
    \end{minipage}
    \label{fig:sushi_elbow_plot}
\end{figure}

Having established $g=4$ as our working configuration, we now examine the distinct preference structures identified by the model. Table~\ref{tab:sushi_g4} details the estimated relative weight and the corresponding consensus ranking for each of the four groups.

\begin{table}[ht!]
\centering
\caption{Consensus rankings and estimated weights for the four identified preference groups ($g=4$) in the SUSHI dataset.}
\label{tab:sushi_g4}
\scriptsize
\begin{tabular}{c l l l l}
\toprule
Group & 1 & 2 & 3 & 4 \\
\midrule
Weight & $0.407$ & $0.309$ & $0.177$ & $0.107$ \\
\midrule
1  & Fatty tuna    & Fatty tuna    & Sea urchin    & Salmon roe    \\
2  & Tuna          & Sea urchin    & Cucumber roll & Shrimp        \\
3  & Tuna roll     & Egg           & Tuna          & Sea eel       \\
4  & Squid         & Sea eel       & Sea eel       & Tuna          \\
5  & Salmon roe    & Shrimp        & Shrimp        & Sea urchin    \\
6  & Shrimp        & Salmon roe    & Salmon roe    & Fatty tuna    \\
7  & Sea eel       & Squid         & Squid         & Tuna roll     \\
8  & Egg           & Tuna roll     & Tuna roll     & Cucumber roll \\
9  & Cucumber roll & Tuna          & Egg           & Egg           \\
10 & Sea urchin    & Cucumber roll & Fatty tuna    & Squid         \\
\bottomrule
\end{tabular}
\end{table}

Rather than detailing each ranking exhaustively, a macroscopic view of these four components reveals distinct and competing preference patterns. The largest profile (Group 1, 40.7\%) represents a mainstream, tuna-centric consensus, strongly favoring \emph{Fatty tuna}, \emph{Tuna}, and \emph{Tuna roll}. Group 2 (30.9\%) shares the top preference for \emph{Fatty tuna} but notably elevates \emph{Sea urchin}, \emph{Egg}, and \emph{Sea eel} (\emph{Anago}), suggesting a more traditional, Tokyo-style (\emph{Edomaezushi}) signal. In stark contrast, Group 3 (17.7\%) acts as a compensatory counterweight to the dominant trend; it places \emph{Fatty tuna} at the very bottom and instead favors \emph{Sea urchin} and \emph{Cucumber roll}. Finally, Group 4 (10.7\%) captures an alternative seafood-oriented structure led by \emph{Salmon roe} and \emph{Shrimp}. 

Since the SUSHI dataset was explicitly created to evaluate clustering methods, the existence of multiple preference profiles is expected. However, the significant contribution here is that the MLOP successfully uncovers this rich, competing heterogeneity by operating exclusively on a highly compressed $10 \times 10$ aggregate matrix. This makes it worth positioning our framework with respect to probabilistic approaches in the literature, such as Mixtures of Mallows Models \citep{lu2014effective, vitelli2018probabilistic}. Both approaches aim to identify heterogeneous preference groups and have been successfully applied to the complete individual rankings of the SUSHI dataset, but they rely on fundamentally different methodologies. Probabilistic models estimate distributions over the space of rankings and therefore require access to the full set of 5,000 individual profiles. In contrast, the MLOP is a deterministic combinatorial approach proving that a simple pairwise aggregate representation can still retain enough structural information to recover interpretable preference patterns. Consequently, the proposed framework is especially useful in settings where individual-level rankings are unavailable, computationally expensive to process, or restricted for privacy reasons.

\section{Conclusions and future research} \label{sec:CFR}

In this paper, we have addressed the limitations of the classical Linear Ordering Problem (LOP) when applied to heterogeneous populations. By implicitly assuming a single underlying preference structure, the standard LOP can lead to the structural marginalization of minority groups, producing a single consensus that may contradict the preferences of significant sub-populations. To overcome this, we introduced an extension of the problem that partitions the population into a predefined number of latent groups, jointly estimating their specific linear orders and relative weights directly from an aggregated pairwise comparison matrix. 

From a methodological standpoint, we propose exact mixed-integer programming (MIP) formulations for this novel problem, including a compact reformulation that provides a clear geometric interpretation within the linear ordering polytope. We also develop a multi-start alternating-direction matheuristic, validating both through computational experiments.

The practical value of our framework was validated using the real-world SUSHI preference dataset, a well-known benchmark for preference clustering. Using the elbow criterion, we set $g=4$ and decomposed the aggregated matrix into distinct social profiles. Rather than yielding a single consensus, the method identified a dominant preference for fatty tuna alongside meaningful minority patterns (e.g., sea urchin, sea eel, or alternative seafood). Notably, this structural detail is obtained from a compressed $10 \times 10$ pairwise comparison matrix, whereas probabilistic approaches require thousands of individual rankings to capture similar heterogeneity.

The results of this study suggest several directions for future research. From an algorithmic perspective, developing tailored exact methods (e.g., Branch-and-Cut based on a deeper polyhedral analysis) and enhancing the matheuristic with frameworks such as Iterated Local Search or Variable Neighborhood Search could improve scalability and robustness.

From a modeling viewpoint, extending the framework to determine the number of groups $g$ endogenously, for instance via a complexity-penalized objective, would be valuable. Additionally, handling incomplete pairwise comparison matrices or allowing partial orders would broaden applicability to settings with missing or uncertain data.

% The results presented in this study open several realistic avenues for future research. First, from an algorithmic perspective, developing tailored exact approaches, such as Branch-and-Cut algorithms based on a deeper polyhedral study of the multiple-group linear ordering polytope, could help close the optimality gap for larger instances. Furthermore, enhancing the matheuristic with more sophisticated frameworks, such as Iterated Local Search (ILS) or Variable Neighborhood Search (VNS), could further improve computational efficiency and solution robustness.

% Second, from a modeling perspective, the current framework assumes that the number of groups $g$ is fixed a priori. Integrating a mechanism to automatically determine the optimal $g$ during the optimization process, perhaps by introducing a penalty term for model complexity in the objective function, would be a valuable extension. Finally, extending the current formulations to handle incomplete pairwise comparison matrices or to output partial orders rather than strict linear orders would broaden the applicability of the model to a wider range of real-world decision-making scenarios with missing information.

\section*{Acknowledgments}
{\sloppy
C. Domínguez and M. Landete were supported by project PID2021-122344NB-I00, funded by MCIN/AEI/10.13039/501100011033 and by “ERDF A way of making Europe”. C. Domínguez was also supported by project PID2024-156594NB-C21, funded by MCIU/AEI/10.13039/501100011033 and by ERDF/EU, and by Generalitat Valenciana under project Emergent CIGE/2024/132. J.D. Jaime-Alcántara and M. Landete were also supported by grant CIPROM/2024/34, funded by the Conselleria de Educación, Cultura, Universidades y Empleo, Generalitat Valenciana. J.A. Aledo was supported by SBPLY/21/180225/000062 (Junta de Comunidades de Castilla-La Mancha and ERDF A way of making Europe), PID2022-139293NB-C32 (MICIU/AEI/10.13039/501100011033 and ERDF, EU) and 2025-GRIN-38476 (Universidad de Castilla-La Mancha and ERDF A way of making Europe).\par}

\bibliographystyle{cas-model2-names}

\bibliography{references}

\end{document}